\documentclass{amsart}
\usepackage{amsmath,latexsym}
\usepackage[psamsfonts]{amssymb}
\usepackage{times}
\usepackage[mathcal]{euscript}
\usepackage{amsfonts}
\usepackage{amssymb}
\usepackage{amsmath}
\usepackage{amsthm}
\usepackage{bm}
\usepackage[english]{babel}
\usepackage[bookmarks]{hyperref}
\numberwithin{equation}{section} \textwidth=140mm \textheight=200mm
\newcommand{\bbR}{\mathbb R}

\newcommand{\bbZ}{\mathbb Z}

\renewcommand{\epsilon}{\varepsilon}

\newcommand{\be}{\begin{equation}}
\newcommand{\ee}{\end{equation}}
\newcommand{\no}{\nonumber}

\newcommand{\C}{\mathbb{C}}

\newcommand{\N}{\mathbb{N}}

\newcommand{\R}{\mathbb{R}}

\newcommand{\T}{\mathbb{T}}

\newcommand{\Z}{\mathbb{Z}}

\newcommand{\cB}{{\mathcal B}}
\newcommand{\cC}{{\mathcal C}}

\newcommand{\cE}{{\mathcal E}}
\newcommand{\cF}{{\mathcal F}}

\newcommand{\cH}{{\mathcal H}}

\newcommand{\cM}{{\mathcal M}}

\newcommand{\cR}{{\mathcal R}}

\newcommand{\cU}{{\mathcal U}}

{\bf}{\it}
\newtheorem{theorem}{Theorem}[section]
\newtheorem{lemma}[theorem]{Lemma}
\newtheorem{corollary}[theorem]{Corollary}
\newtheorem{hypothesis}[theorem]{Hypothesis}
\newtheorem{definition}[theorem]{Definition}
\newtheorem{proposition}[theorem]{Proposition}

\newtheorem{remark}[theorem]{Remark}

\date{\today}

\begin{document}

\title[Threshold effects of the two-particle Schr\"odinger operators...] {Threshold effects of the two-particle Schr\"odinger
operators on lattices}

\author{Saidakhmat N. Lakaev $^{1}$}
\address{
$^1$ Samarkand State University, Samarkand (Uzbekistan)}
\email{slakaev@mail.ru}
\author{Volker Bach $^{2}$}
\address{$^2$ Institut fuer Analysis und Algebra
Carl-Friedrich-Gauss-Fakultaet Technische Universitaet Braunschweig
(Germany)} \email{v.bach@tu-bs.de}
\author{W.~de Siqueira Pedra$^{3}$}
\address{
$^3$ Departamento de Fisica Matemetica do Instituto de Fisica,
Universidade de Sao Paulo, 05314-970 Sao Paulo, SP, Brazil}
\email{wpedra@if.usp.br}

\maketitle
\begin{abstract}
We consider a wide class of the two-particle Schr\"{o}dinger operators $%
H_{\mu}(k)=H_{0}(k)+\mu V, \,\mu>0,$ with a fixed two-particle quasi-momentum $k$ in the $d$%
-dimensional torus $\mathbb{T}^d$, associated to the Bose-Hubbard
hamiltonian $H_{\mu}$ of a system of two identical quantum-mechanical particles (bosons) on the
$d$- dimensional hypercubic lattice $\mathbb{Z}%
^d$ interacting via short-range pair potentials. We study the
existence of eigenvalues of $H_{\mu}(k)$ below the threshold of the
essential spectrum depending on the interaction energy $\mu>0$ and
the quasi-momentum $k\in \mathbb{T}^d$ of particles. We prove that
the threshold (bottom of the essential spectrum), as a singular
point (a threshold resonance or a threshold eigenvalue), creates
eigenvalues below the essential spectrum under perturbations of both
the coupling constant $\mu>0$ and the quasi-momentum $k$ of the
particles. Moreover, we show that if the threshold is a regular
point, then it does not create any eigenvalues under small
perturbations of the coupling constant $\mu>0$ and the
quasi-momentum $k$.
\end{abstract}


\section{Introduction}

The main goal of the present paper is to give a thorough
mathematical treatment of the spectral properties for the two-particle lattice Schr\"{o}dinger operators $%
H_{\mu}(k)=H_{0}(k)+\mu V, \,\mu>0,$ where $k\in \mathbb{T}^d$ being the two-particle quasi-momentum in the $d$%
-dimensional torus, associated to the Bose-Hubbard hamiltonian
$H_{\mu}$ of a system of two identical quantum-mechanical particles (bosons) on
the $d$-dimensional hypercubic lattice $\mathbb{Z}%
^d$ interacting via short-range pair attractive potentials, with an
emphasis on {\it new threshold phenomena}  that are not present in
the continuous case (see, e.g., \cite{ALMM2006}, \cite{BdsPL2018},
\cite{FariaIorattiCarrol2002}, \cite{GrafSchenker1997},
\cite{Lakaev1992}--\cite{Mat}, \cite{Mog} for relevant discussions).

Throughout physics, stable composite objects are usually formed by
way of attractive forces, which allows the constituents to lower
their energy by binding together. Repulsive forces separate
particles in the free space. However, in a structured environment
such as a periodic potential and in the absence of dissipation,
stable composite objects can exist even for repulsive interactions
that arise from the lattice band structure \cite{Winkler}.

The Bose-Hubbard model, which is used to describe the attractive and
repulsive pairs is the theoretical basis for applications. The work
\cite{Winkler} exemplifies the important correspondence between the
Bose-Hubbard model \cite{ Bloch,Jaksch_Zoller} and atoms in optical
lattices, and helps to pave the way for many more interesting
developments and applications \cite{Thalhammer}. In particular, the
dynamics of  ultracold atoms loaded in the lower band is well
described by the Bose-Hubbard hamiltonian.

In the continuum space (continuum), due to rotational invariance,
the multi-particle hamiltonian separates into a free hamiltonian for
the center of mass and a hamiltonian for the relative motion. Bound
states are eigenstates of the latter hamiltonian.

The fundamental difference between the continuum and the lattice
multi-particle hamiltonian is that in lattice, the free hamiltonian
is not rotationally invariant. Therefore, in contrast to the
continuum, in the lattice there is \textit{excess mass} phenomenon,
i.e., the effective mass of the bound state of an $N$-particle
system is, in general, strictly greater than the sum of the
effective masses of the constituent quasi--particles. This has been
discussed, e.g., in  \cite{Mat}, \cite{Mog}.

For the Schr\"odinger operators with short-range potentials on
$\bbR^3$ and $\bbZ^3$ and on perturbations of Schr\"odinger
operators with periodic potentials  one observes the emission of
negative bound states from the continuous spectrum at the so-called
{\it critical potential strength} (see, e.g., \cite{AGH1982},
\cite{AGHH2004}, \cite{ALMM2006}, \cite{FassariKlaus},
\cite{KlausSimon1980}, \cite{Lakaev1992}, \cite{Lakaev1993},
\cite{LakaevTilavova1994}, \cite{Rauch1980}, \cite{Yafaev1974},
\cite{Yafaev1979}). This phenomenon is closely related to the
existence of gene\-ralized eigenfunctions, which are, vanishing at
infinity, solutions of the Schr\"odinger equation with zero energy,
but are not square integrable. These solutions are usually called
zero-energy resonance functions and, in this case, the Hamiltonian
is  called a critical one and the Schr\"odinger operator is said to
have a  zero-energy resonance (virtual level).

The appearance of the negative eigenvalues for the critical
(non-negative) two--particle Schr\"{o}\-dinger operators on $\bbR^3$
under infinitesimally negative perturbations, i.e., the presence of
the threshold resonances, leads to the existence of infinitely many
bound states (Efimov's effect) for the corresponding three-particle
system (see, e.g., \cite{AHW1971}, \cite{Efimov1970},
\cite{OvchinnikovSigal1989}, \cite{Sobolev1993}, \cite{Tamura1991},
\cite{Tamura1993}, \cite{Yafaev1974}).

The threshold of the essential spectrum for the Schr\"{o}dinger
operator is either regular or singular point (threshold resonance or
threshold eigenvalue). In the continuum, whether the threshold is a
regular or a singular point for the two-particle Schr\"{o}dinger
operator depends only on the interaction $\mu V$. In the lattice
case it depends not only on the interaction $\mu V$ of particles,
but also to the {\it quasi-momentum} $k$  of the particle pair.

In the current paper, we study the emission mechanisms of
eigenvalues
at the thresholds of the essential spectrum for the  operator $%
\widehat{H}_{\mu}(k)$, depending on the interaction $\mu \widehat{V}$ and the quasi-momentum $k\in \mathbb{T}%
^{d}$, for all dimensions $d\geq 3$. We prove the existence of bound
states in the following two cases:

(\textit{low} $d$) In the case $d=1,2$, for any nonzero potential
$\mu \widehat{V}$ and each quasi-momentum $k\in \mathbb{T}^{d}$ of
the particle pair (Theorem \ref{existence} and \ref{Inequalities}).

(\textit{high} $d$) In the case $d\geq 3$, for \textit{large}
potentials $\mu \widehat{V}$ and for each $k\in \mathbb{G}$, where
$\mathbb{G}\subset\mathbb{T}^{d}$ is a region that includes the
point $0\in\mathbb{T}^{d}$ (Theorem \ref{existence(d>3)} and Remark
\ref {remark(d>3)}).

We establish the existence of eigenvalues $E_{\mu}(k)$ of
$\widehat{H}_{\mu}(k)$ for each non-zero $k\in \mathbb{T}^{d}$
provided that the bottom of the essential spectrum of
$\widehat{H}_{\mu}(0)$ is its threshold resonance or threshold
eigenvalue, by applying a generalization of the well known
Birman-Schwinger principle (see, e.g., \cite{Lieb2004},
\cite{Simon1979}). The method we use gives implicit forms of the
eigenfunctions (bound states) by means of eigenfunctions of the
generalized Birman-Schwinger operator.

We find a two-sided assessments for the eigenvalue $E(k)$ of
$\widehat{H}_{1}(k)$ depending on the quasi-momentum $k\in
\mathbb{T}^{d}$ by means of ${\mathcal{E}}_{\min}(k)$, the threshold
of the essential spectrum of $\widehat{H}_{1}(k)$ and the eigenvalue
$E(0)$ of $\widehat{H}_{1}(0)$ (Corollary \ref{Inequality3}). This
result shows that for any nonzero $k\in \mathbb{T}^{d}$ the
eigenvalue $E(k)$ is strictly greater than $E(0)$ and lies below the
threshold ${\mathcal{E}}_{\min}(k)$.

We prove, for $d\ge3$, that  the threshold
${\mathcal{E}}_{\min}(k)$, as a singular point, i.e., as a threshold
resonance or a threshold eigenvalue, creates eigenvalues below the
essential spectrum under small perturbations of both the effective
mass (by changing the quasi-momentum of the particles) and the
coupling constant $\mu>0$. However, if the threshold
${\mathcal{E}}_{\min }(k)$ is a regular point, then no eigenvalues
are created under such perturbations.

Moreover, we show that for each $k_{0}\in \mathbb{G}$, there exists
a neighborhood $\mathrm{G}(k_{0})\subset \mathbb{G}$ and a manifold
$\mathbb{M}_{=}(k_{0}) \subset \mathbb{G}$  of codimension one, such
that for each $k\in \mathbb{M}_{=}(k_{0})$ the threshold
${\mathcal{E}}_{\min }(k)$ is a singular point of
$\widehat{H}_{\mu}(k)$, but it is a regular point if $k\in
\mathbb{G}\setminus \mathbb{M}_{=}(k_{0})$. Furthermore, if the
threshold ${\mathcal{E}}_{\min }(k_{0})$ of
$\widehat{H}_{\mu}(k_{0})$: \,\text{(i)} is a regular point, then
for each $k$ lying in a neighborhood of $k_{0}\in \mathbb{G}$, the
number of eigenvalues of $\widehat{H}_{\mu}(k)$ lying below the
threshold $\cE_{\min}(k)$ remains unchanged (Theorem
\ref{regular_point}). \,\text{(ii)} is a singular point, then there
is an open set $\mathbb{M}_{>}(k_{0})\subset \mathbb{G}$ with
$k_{0}\in \overline{\mathbb{M}_{>}(k_{0})}$ such that for all $k\in
\mathbb{M}_{>}(k_{0})$ the operator $\widehat{H}_{1}(k)$ has an
eigenvalue below ${\mathcal{E}}_{\min }(k)$ (Theorem
\ref{singular_point}).

We observe that for $k_0=0\in \T^3$, the case \text{(i)} above
yields Efimov's effect for the three-particle lattice
Schr\"{o}dinger
operators $\mathit{\mathbf{H}}%
(K),\,K\in\T^d$, associated to the Bose-Hubbard hamiltonian of a
system of three particles on $\mathbb{Z}^{3}$ interacting via short-range pair potentials: the operator $\mathit{\mathbf{H}}%
(0)$ has infinitely many eigenvalues below the bottom of the three
particle essential spectrum, whereas for all non-zero $K\in
\mathbb{T}^{3}$ close to $K=0$, the operator $\mathit{\mathbf{H}}%
(K)$ has finitely many eigenvalues (see, e.g.,
\cite{AbdullaevLakaev2003}, \cite{ALzM2004}, \cite{ALK2012},
\cite{Lakaev1993}).

It is also shown that for $d\ge3$, the total number of both the
eigenvalues (counting multiplicities) below the threshold of
essential spectrum and the multiplicity of the {\it singular point}
${\mathcal{E}}_{\min }(k)$ is a non-decreasing function of the
quasi-momentum $k\in \mathbb{T}^{d}$ and $\mu>0$ (Corollary
\ref{greater} and \ref{total_number}).

We note that unlike the case of Schr\"{o}%
dinger operators in $\mathbb{R}^{d}$, the lattice Schr\"{o}dinger
operators may have eigenvalues above the threshold of the essential
spectrum, when the sign of the potential is changed. The repulsive
case can be investigated exactly in the same way as the attractive
one treated here.

The paper is organized as follows: In Section 2 we introduce the
lattice two-particle hamiltonians, decompose them into the von
Neumann direct integrals and introduce the Schr\"{o}dinger operators
$H_{\mu}(k)$ with fixed two-particle quasi-momentum $k\in
\mathbb{T}^{d}$.  In Section 3 we prove a generalized
Birman-Schwinger principle to introduce notions of \textit{singular
and regular points} of the essential spectrum and to study threshold
effects for the lattice Schr\"{o}dinger operators
$\,\widehat{H}_{\mu}(k)$. The main results of the paper are stated
in Section 4 and proved in Section 5.

\section{The two-particle hamiltonians on lattices}
\subsection{The two-particle hamiltonian -- position space representation}

Let $\mathbb{Z}^{d}$ be the $d$-dimensional hypercubic lattice and
$\Z^{d}\otimes\Z^{d}$ be cartesian product.

Let $\ell ^{2}(\Z^{d}\otimes\Z^{d})$ be the Hilbert space of
square-summable functions on $\Z^{d}\otimes\Z^{d}$  and $\ell
^{2,s}(\Z^{d}\otimes\Z^{d})\subset\ell ^{2}(\Z^{d}\otimes\Z^{d})$ be
the subspace of symmetric functions.

The free hamiltonian $\widehat{\mathbb{H}}_{0}$ of a system of two
identical particles (bosons), in the position space representation,
is usually associated with the following self-adjoint (bounded)
multidimensional Toeplitz-type operator on the Hilbert space $\ell
^{2,s}(\Z^{d}\otimes\Z^{d})$ (see, e.g., \cite{Mat}):
\begin{equation}
(\widehat{\mathbb{H}}_{0}\hat{f})(x_1,x_2)= \sum_{s_1,s_2\in\Z^d}
[\hat{\varepsilon}(x_1-s_1)+\hat{\varepsilon}(x_2-s_2)]\hat{f}(s_1,s_2),\,\,\hat{f}\in\ell
^{2,s}(\Z^{d}\otimes\Z^{d}), \label{two_free}
\end{equation}
where $\hat{\varepsilon}\in \ell ^{1}(\mathbb{Z}^{d})$ is a real
valued even function.

The interaction operator $\widehat{\mathbb{V}}$ of two bosons, in
the position space representation, is the multiplication operator by
the non-positive function $\hat{v}\in
\ell^{1}(\mathbb{Z}^{d};\mathbb{R}_{0}^{-})$ , i.e.,
\begin{equation}\label{interaction}
(\widehat{\mathbb{V}}\hat{f})(x_1,x_2)=\hat{v}(x_1-x_2)\hat{f}(x_1,x_2),\,\,\hat{f}\in\ell
^{2,s}(\Z^{d}\otimes\Z^{d}).
\end{equation}%

The total hamiltonian $\widehat{\mathbb{H}}_{\mu}$  of a system of
two identical particles (bosons) with the pair non-positive
interaction $\hat{v}\in
\ell^{1}(\mathbb{Z}^{d};\mathbb{R}_{0}^{-})$, in the position space
representation, is associated with the bounded self--adjoint
operator on $\ell ^{2,s}(\Z^{d}\otimes\Z^{d})$:
\begin{equation}\label{position}
\widehat{\mathbb{H}}_{\mu}=\widehat{\mathbb{H}}_{0}+\mu\widehat{\mathbb{V}},\,\,
\mu>0.
\end{equation}%

\subsection{The two-particle hamiltonian -- momentum space representation}
Let $\mathbb{T}^{d}=(\mathbb{R}/2\pi \mathbb{Z)}^{d}\equiv \lbrack
-\pi ,\pi )^{d}$ be the $d$- dimensional torus, the Pontryagin dual
group of $\mathbb{Z}^{d}$, equipped with the (normalized) Haar
measure
\begin{equation*}
\eta (\mathrm{d}p)=\frac{\mathrm{d}^{d}p}{(2\pi )^{d}}\text{ }.
\end{equation*}
Let \,$L^2(\T^d,\eta)$ be the Hilbert space of square-integrable
functions on $\T^d$ %
and ${\mathcal{F}}:\ell^{2}(\mathbb{Z}^{d})\rightarrow
L^2(\T^d,\eta)$ be the Fourier transform
\begin{equation}\label{potential}
{f}(p):=({\mathcal{F}}\hat{f})(p)=\sum_{x\in\Z^d}e^{
\mathrm{i}(p,x)}\hat{f}(x)
\end{equation}
and ${\mathcal{F}}^{\ast}$ is its inverse
\begin{equation}
\lbrack {\mathcal{F}}^{\ast} {f}](x)\ :=\ \int_{\mathbb{T}^{d}}e^{-%
\mathrm{i}(p,x)}{f}(p)\,\eta (\mathrm{d}p)\text{ }.
\label{Fourierinvers}
\end{equation}

The free hamiltonian $
{\mathbb{H}}_{0}=(\mathcal{F}\otimes \mathcal{F})\widehat{\mathbb{H}}_{0}\mathbb{(}%
\mathcal{F}^{\ast}\otimes \mathcal{F}^{\ast}) $ of a system of two
identical particles (bosons), in the momentum space representation,
is the multiplication operator by the function ${\varepsilon}
(k_{1})+{\varepsilon}(k_{2})$ on the Hilbert space
$L^{2,s}(\mathbb{T}^{d}\otimes\mathbb{T}^{d},\eta\otimes\eta)$ of
symmetric functions on the cartesian product
$\mathbb{T}^{d}\otimes\mathbb{T}^{d}$ of the torus $\mathbb{T}^{d}$:
\begin{equation}
({\mathbb{H}}_{0}{f})(k_{1},k_{2})=({\varepsilon}%
(k_{1})+{\varepsilon}(k_{2})){{f}}(k_{1},k_{2}),
\end{equation}%
where the continuous function (dispersion relation) $\varepsilon$ is
given by
\begin{equation}\label{dispersion}
{\varepsilon}(p)=[{\mathcal{F}}\hat{\varepsilon}](p)=\sum_{x\in\Z^d}e^{
\mathrm{i}(p,x)} \hat{\varepsilon}(x). \end{equation}

The interaction operator  ${\mathbb{V}}=(\mathcal{F}\otimes \mathcal{F}%
)\widehat{\mathbb{V}}(\mathcal{F^{\ast}}\otimes \mathcal{F}^{\ast})$
is the integral operator of convolution type acting in
$L^{2,s}(\mathbb{T}^{d}\otimes\mathbb{T}^{d},\eta\otimes\eta)$ as
\begin{equation}\label{Interaction1}
({\mathbb{V}}{{f}})(k_{1},k_{2})={\int\limits_{%
\mathbb{T}^{d}}}{v}(k_{1}-q){{f}}(q,k_{1}+k_{2}-q)\eta (%
\mathrm{d}q),
\end{equation}%
where the kernel function ${v}(\cdot )$ is given by
\begin{equation}\label{potential}
{v}(p)=[{\mathcal{F}}\hat{v}](p)=\sum_{x\in\Z^d}e^{ \mathrm{i}(p,x)}
\hat{v}(x).
\end{equation}

The total two-particle hamiltonian ${\mathbb{H}}_{\mu}$ of a system
of two identical quantum-mechanical particles (bosons) interacting
via a short range attractive potentials $\hat{v}$, in the momentum
space representation, is the bounded self--adjoint operator acting
in $L^{2,s}(\mathbb{T}^{d}\otimes\mathbb{T}^{d},\eta\otimes\eta)$ as
\begin{equation}\label{momentum}
{\mathbb{H}}_{\mu}={\mathbb{H}}_{0}+\mu{\mathbb{V}}\text{}.
\end{equation}%

\subsection{Decomposition of the two--particle hamiltonians into the von Neumann
direct integrals}

Let $k=k_{1}+k_{2}\in \mathbb{T}^{d}$ be the
\textit{quasi--momentum} of a pair of particles. For any fixed $k\in
\mathbb{T}^{d}$, we define the set $\mathbb{Q}_{k}\subset
\mathbb{T}^{d}\otimes\mathbb{T}^{d}$ as
\begin{equation*}
\mathbb{Q}_{k}=\{(q,k-q):q\in \mathbb{T}^{d}\}.
\end{equation*}

We further define the map
\begin{equation*}
\pi :\mathbb{T}^{d}\otimes\mathbb{T}^{d}\rightarrow
\mathbb{T}^{d},\quad \pi ((k_{1},k_{2}))=k_{1}\text{ }.
\end{equation*}%
Denote by $\pi _{k}=\pi |_{\mathbb{Q}_{k}}$, $k\in \mathbb{T}^{d}$
the restriction of $\pi $ on $\mathbb{Q}_{k}\subset
\mathbb{T}^{d}\otimes\mathbb{T}^{d}$. At this point it is useful to
remark that $\mathbb{Q}_{k}$ is a $d$-dimensional manifold
isomorphic to ${\mathbb{T}}^{d}$. The
map $\pi _{k}$ is bijective from $\mathbb{Q}%
_{k}\subset \mathbb{T}^{d}\otimes\mathbb{T}^{d}$ onto
$\mathbb{T}^{d}$ with the inverse
\begin{equation*}
\pi _{k}^{-1}(q)=(q,k-q)\text{ }.
\end{equation*}

Consequently, the Hilbert space
$L^{2,s}(\mathbb{T}^{d}\otimes\mathbb{T}^{d},\eta\otimes\eta)$ can
be decomposed into the von Neumann direct integral as
\begin{equation}
L^{2,s}(\mathbb{T}^{d}\otimes\mathbb{T}^{d},\eta\otimes\eta)\simeq
\int\limits_{k\in {\mathbb{T}}^{d}}^{\oplus }L^{2,e}(\T^d,\eta)\eta
(\mathrm{d}k), \label{decomposSpace}
\end{equation}%
where $L^{2,e}(\T^d,\eta)$ is the Hilbert space of square-integrable
even functions on $\T^d$. The total hamiltonian ${\mathbb{H}}_{\mu}$
of a system of two particles, in the position space representation,
obviously commutes with the representation of the discrete group
$\Z^d$ by shift operators on the lattice, and hence
${\mathbb{H}}_{\mu}$ can be decomposed into the integral (see, e.g.,
\cite{ALMM2006})
\begin{equation}
{\mathbb{H}}_{\mu}\simeq\int\limits_{k\in {\mathbb{T}}^{d}}^{\oplus }{H}_{\mu}%
(k)\eta (\mathrm{d}k)  \label{fiber}.
\end{equation}%
In the physical literature, the parameter $k\in \T^d$ is called
two-particle {\it quasi-momen\-tum} and the corresponding operator
 ${H}_{\mu}%
(k)$ is called Schr\"{o}dinger operator  with {\it fixed
quasi--momentum} $k$.

\subsection{Schr\"{o}dinger operators for particle pairs with fixed
quasi--momentum} For any $\mu>0$ and $k \in {\T}^d$, the
Schr\"{o}dinger operator ${H}_{\mu}%
(k)$ from the decomposition \eqref{fiber}, in the momentum space
representation, is bounded self-adjoint operator acting in
$L^{2,e}(\mathbb{T}%
^{d}, \eta)$ as
\begin{equation}\label{two_schroedinger}
{H}_{\mu}(k)={H}_{0}(k)+\mu{V}.
\end{equation}%
Here the non-perturbed operator ${H}_{0}(k)$, $k\in \mathbb{T%
}^{d}$ is the multiplication operator by the function
${\mathcal{E}}_{k}$ (quasi-momentum-dependent pair dispersion
relation) acting in $L^{2,e}(\mathbb{T}^{d}, \eta)$ as
\begin{equation}
({H}_{0}(k) {f})(p)={\mathcal{E}}_{k}(p)
{f}(p),
\end{equation}%
where
\begin{equation}
{\mathcal{E}}_{k}(p)={\varepsilon}\left( \frac{k}{2}+p\right) +{%
\varepsilon}\left( \frac{k}{2}-p\right) \text{ }
 \label{free_two-particle}
\end{equation}%
and ${\varepsilon}(\cdot)$ is defined in  \eqref{dispersion}.

Note that we identified the torus $\mathbb{T}^{d}$ with $[-\pi ,\pi
)^{d}\subset \mathbb{R}^{d}$ so that the operation change of
variables  is well--defined on $\mathbb{T}^{d}$.

The perturbation operator
 ${V}$ in \eqref{two_schroedinger} is defined
as
\begin{equation}\label{interaction_fixed}
({V} f)(p)=\int\limits_{
\mathbb{T}^{d}}{v}(p-q){f}(q)\eta(%
\mathrm{d}q),\qquad f\in L^{2,e}(\mathbb{T}^{d}, \eta).
\end{equation}%

In the position space representation, the Schr\"{o}dinger operator
$\widehat{H}_{\mu}(k)$ with a fixed quasi--momentum $k\in
{\mathbb{T}}^{d}$ acts in the Hilbert space
$\ell^{2,e}(\mathbb{Z}^{d})$ of all square-summable even functions
on $\Z^d$  as
\begin{equation}\label{two-particle}
\widehat{H}_{\mu}(k)=\widehat{H}_{0}(k)+\mu \widehat{V},\,\,\mu>0,
\end{equation}%
where
\begin{equation} \label{two_particleFree}
(\widehat{H}_{0}(k) \hat{f})(x) \ = \ \sum\limits_{s\in\Z^d} {\hat{\mathcal{E}}}_{k}(x-s)%
\hat{f}(s),\,\, \hat{f}\in\ell^{2,e}(\Z^d),
\end{equation}
\begin{equation}
{\hat{\mathcal{E}}}_{k}(x) :=\ \int\limits_{\mathbb{T}^{d}}e^{-%
\mathrm{i}(p,x)}{\mathcal{E}}_{k}(p)\,\eta (\mathrm{d}p)\text{ }.
\label{Fourierinvers}
\end{equation}
and $\widehat{V}$ in \eqref{two-particle} is defined as
\begin{equation}\label{interact_fix}
(\widehat{V} \hat{f})(x)=\hat{v}(x)\hat{f}(x),\,\,\hat{f}\in\ell
^{2,e}(\mathbb{Z}^{d}).
\end{equation}%

\section{Spectral properties of $\widehat{H}_{\mu}(k)$, $k\in
\mathbb{T}^{d}$}

Since the perturbation operator $\widehat{V}$ is compact, according
to Weyl's theorem \cite[Theorem
XIII.14]{RSIV}, the essential spectrum $\sigma _{\text{ess}}(\widehat{H}_{\mu}(k))$ of the operator $\widehat{H}_{\mu}(k)$, $%
k\in \mathbb{T}^{d}$ coincides with the spectrum ${\sigma
}(H_{0}(k))$ of the non-perturbed operator $H_{0}(k)$. Explicitly,
one has
\begin{equation*}
\sigma _{\text{ess}}(\widehat{H}_{\mu}(k))=[{\mathcal{E}}_{\min
}(k),{\mathcal{E}}_{\max }(k)]\text{ },
\end{equation*}%
with
\begin{equation*}
{\mathcal{E}}_{\min }(k):=\min_{q\in \mathbb{T}^{d}}{\mathcal{E}}%
_{k}(q),\,\, {\mathcal{E}}_{\max }(k):=\max_{q\in \mathbb{T}^{d}}{\mathcal{E%
}}_{k}(q)\text{ }.
\end{equation*}
From the non-positivity of $\widehat{V}$ and the min-max principle,
it follows that all isolated eigenvalues of finite multiplicity lie below ${%
\mathcal{E}}_{\min }(k)$, the bottom  of the essential spectrum $\sigma _{\text{ess}%
}(\widehat{H}_{\mu}(k))$.
\subsection{Birman-Schwinger principle}
Let $d\geq 1$, $k\in \mathbb{T}^{d}$, and ${\mathcal{E}}_{k}(\cdot
)$ be the {\it quasi-momentum-dependent pair dispersion relation}
and $\hat{v}\in\ell^{1}(\mathbb{Z}^{d};\mathbb{R}_{0}^{-})$. For any
$z <{\mathcal{E}}_{\min }(k)$, we define the Birman-Schwinger
operator $\mathbb{B}_{\mu}(k,z)$, which is a non-negative compact
operator acting in $\ell^{2,e}(\mathbb{Z}^{d})$ as
\begin{equation}
\mathbb{B}_{\mu}(k,z)\ :=\
\mu|\widehat{V}|^{1/2}\,\widehat{R}_{0}(k,z)\,|\widehat{V}|^{1/2}\text{
}. \label{def_B-Sch}
\end{equation}%
Here, $\widehat{R}_{0}(k,z)$,  $z\in \mathbb{C}\setminus \lbrack
{\mathcal{E}}_{\min }(k),{\mathcal{E}}_{\max }(k)]$ is the resolvent
of the operator $\widehat{H}_{0}(k)$ and
$|\widehat{V}|^{\frac{1}{2}}$ is the (unique) positive square root
of the (positive) operator $|\widehat{V}|$:
\begin{equation}
 (|\widehat{V}|^{\frac{1}{2}}\hat{\psi})(x)={|\hat{v}(x)|}^{\frac{1}{2}}\hat{\psi}
(x),\quad \hat{\psi}\in \ell ^{2,e}(\mathbb{Z}^{d})\text{ }.
\label{root_v}
\end{equation}

The kernel function ${\mathcal{B}}_{\mu}(k,z;\cdot ,\cdot )$, $k\in \mathbb{T}^{d}$%
, $z<{\mathcal{E}}_{\min }(k)$, associated to the Birman-Schwinger operator $%
\mathbb{B}_{\mu}(k,z)$ is given by
\begin{equation}
{\mathcal{B}}_{\mu}(k,z;x,y)=\mu{|\hat{v}(x)|}^{\frac{1}{2}}\widehat{\mathcal{R}}_{0}(k,z;x-y)
{|\hat{v}(y)|}^{\frac{1}{2}},\,\text{\ \ }x\,,y\in \mathbb{Z}^{d},
\label{kernel}
\end{equation}%
where
\begin{equation}
\widehat{\mathcal{R}}_{0}(k,z;x):=\int\limits_{\mathbb{T}^{d}}\frac{e^{i(q,x)}}{{\mathcal{E}}%
_{k}(q)-z}\eta (\mathrm{d}q),\,\text{\ \ }x\in \mathbb{Z}^{d}.
\label{fouriertransform}
\end{equation}
Now we recall, for the convenience of the reader, the well known
Birman-Schwinger principle (see, e.g., \cite{Lieb2004}, p.180 and
\cite{Simon1979}, p. 89.),  associated to the Schr\"{o}dinger
operator $\widehat{H}_{\mu}(k)$,  the proof of which can be found in
\cite{BdsPL2018}).

\begin{proposition}{\bf (Birman-Schwinger principle)} \label{BS}
Let $k\in\T^d$ and $z<\cE_{\min }(k)$.
\begin{enumerate}
\item[(i)] If $\hat{f}\in \ell^{2,e}(\Z^d)$ solves $\widehat{H}_{\mu}(k)
\hat{f} = z\hat{f}$, then $\hat{\psi}:=
\mu|\widehat{V}|^{1/2}\hat{f}\in \ell^{2,e} (\Z^d)$ solves
$\hat{\psi}=\mathbb{B}_{\mu}(k,z)\hat{\psi}$.
\item[(ii)] If $\hat{\psi}\in \ell^{2,e} (\Z^d)$ solves
$\hat{\psi} = \mathbb{B}_{\mu}(k,z)\hat{\psi}$, then $\hat{f}:=
\widehat{R}_0(k,z)|\widehat{V}|^{1/2}\hat{\psi} \in
\ell^{2,e}(\Z^d)$ solves $\widehat{H}_{\mu}(k)\hat{f} = z \hat{f}$.
\item[(iii)] The number $z$ is an eigenvalue of
$\widehat{H}_{\mu}(k)$ of multiplicity $m$ if and only if $1$ is an
eigenvalue for \, $\mathbb{B}_{\mu}(k,z)$ of multiplicity $m$.
\item[(iv)]Counting multiplicities, the number $\mathcal{N}_{-}(z,\widehat{H}_{\mu}(k))$ of
eigenvalues of \,$\widehat{H}_{\mu}(k)$ smaller than $z$ equals to
the number \,$\mathcal{N}_{+}(1,\mathbb{B}_{\mu}(k,z))$\, of
eigenvalues of \,$\mathbb{B}_{\mu}(k,z)$\, greater than $1,$
\mbox{i.e.,}
\begin{equation} \label{BS_equality}
\mathcal{N}_{-}(z,\widehat{H}_{\mu}(k))
=\mathcal{N}_{+}(1,\mathbb{B}_{\mu}(k,z)).
\end{equation}
\end{enumerate}
\end{proposition}
\begin{lemma}\label{kernel_BS}
Let $\hat{v} \in \ell^1(\Z^d; \R_0^-)$ be a non-positive function.
Then for any $x,y\in\Z^d$, the function $\cB_{\mu}(k,\cdot;x,y)$ is
analytic in $z\in \mathbb{C}\setminus \lbrack {\mathcal{E}}_{\min
}(k),{\mathcal{E}}_{\max }(k)]$.
\end{lemma}

\begin{proof}
The equality \eqref{kernel} and the analyticity of the function
$\widehat{\cR}_0(k,z;x)$ in $z\in \mathbb{C}\setminus \lbrack
{\mathcal{E}}_{\min }(k),{\mathcal{E}}_{\max }(k)]$ yield the proof
of Lemma \ref{kernel_BS}.
\end{proof}

\subsection{A generalization of the Birman-Schwinger principle}

In what follows we assume the following Hypotheses \ref{hypothesis}
on ${\varepsilon}(\cdot )$ and  $\hat{v}$.

\begin{hypothesis}\label{hypothesis}
(i) The function ${\varepsilon}(\cdot )\in C^{(3)}(%
\mathbb{T}^{d}),\,d\ge1$ is a real--valued even  Morse function and
has a unique minimum at $0\in\T^d$.

(ii) The function $\hat{v}$ is absolutely summable and non-positive,
i.e., $\hat{v}\in \ell
^{1}(\mathbb{Z}^{d};\mathbb{R}_{0}^{-}),\,d\ge1$.
\end{hypothesis}
We need these assumptions in order to introduce a
\emph{generalization} of the Birman--Schwinger operator and the
concept of a \textit{threshold resonance (virtual level)} and a
\textit{threshold eigenvalue} for the Schr\"{o}dinger operator
$\widehat{H}_{\mu}(k)$, $k\in \mathbb{T}^{d}$. By
\emph{generalization}, we mean that the definition of the
Birman--Schwinger operator $\mathbb{B}_{\mu}(k,z),z
<{\mathcal{E}}_{\min }(k)$ is extended to the case $z
={\mathcal{E}}_{\min }(k)$, the bottom of the
essential spectrum of $\widehat{H}_{\mu}(k)$. We recall that such a generalization is only meaningful for $%
d\geq 3$.

In addition, in this subsection we give some properties  of the
\emph{generalized Birman-Schwinger operator} (Lemma \ref{kernelBS}
and Theorem \ref{General_BS}), which will be proved in Section 4.

The parametrical Morse lemma (see, e.g., \cite{Fedoryuk1987}, p.
113) yields the following lemma.
\begin{lemma}\label{region}
(i) The minimum ${\varepsilon}_{\min}=\min\limits_{p\in \mathbb{T}^{d}}%
{\varepsilon}(p)$ is attained at the point $p=0\in \mathbb{T}^{d}$.

(ii) There is a maximal open neighborhood $\mathbb{G%
}\subset \mathbb{T}^{d}$ containing the point $0\in \mathbb{T}^{d}$
such that, for any $k\in \mathbb{G}$, the pair dispersion relation
${\mathcal{E}}_{k}(\cdot )\in
C^{(3)}(\mathbb{T}^{d})$ is a Morse function and  for each $k\in \mathbb{%
G}$,\,  the minimum ${\mathcal{E}}_{\min }(k)=\min\limits_{q\in
\mathbb{T}^{d}}{\mathcal{E}}_{k}(q)$ is attained at the point $%
p(k)=0\in \mathbb{T}^{d}$.
\end{lemma}

The proof of Lemma \ref{region} can be found in \cite[Lemma
3]{Lakaev1992}.

\begin{remark}\label{conditionallly}
The following subclass of the one-particle systems is of certain
interest (see, e.g., \cite{CL1990}). It is introduced by the
additional requirement that the dispersion relation
${\varepsilon}(p)$ is a real-valued continuous conditionally
negative definite function and hence
  \begin{itemize}
  \item[(i)] ${\varepsilon} $ is an even function,
  \item[(ii)] ${\varepsilon}(p)$ has a minimum at $p=0$.
  \end{itemize}
Recall (see, e.g., \cite{ALMM2006}, \cite{RSIV}) that a
complex-valued bounded function ${\varepsilon}:\T^d\longrightarrow
\C$ is called conditionally negative definite if
${\varepsilon}(p)=\overline{{\varepsilon}(-p)}$  and
\begin{equation}\label{nn}
  \sum_{i,j=1}^{n}{\varepsilon}(p_i-p_j)z_i\bar z_j\le 0
 \end{equation}
for any $n\in \N$, $p_1, p_2, .., p_n\in \T^d$ and for all ${\bf
z}=(z_1, z_2, ..., z_n)\in \C^n$ satisfying $\sum_{i=1}^nz_i=0$.
\end{remark}

\begin{definition}\label{remark_2}
Assume $d\geq 3$ and  Hypothesis \ref {hypothesis}. Let
\begin{equation}
\widehat{\mathcal{R}}_{0}(k,{\mathcal{E}}_{\min }(k);x):=\int\limits_{\mathbb{T}^{d}}\frac{%
e^{i(q, x)}}{{\mathcal{E}}_{k}(q)-{\mathcal{E}}_{\min }(k)}\eta (\mathrm{d}%
q),\, x\in \mathbb{Z}^{d}.  \label{resolvent}
\end{equation}
For $k\in \mathbb{G}%
\subset \mathbb{T}^{d}$, we define the generalized Birman-Schwinger operator $\mathbb{B}_{\mu}(k,{\mathcal{E}}%
_{\min}(k))$ by means of the kernel function
\begin{equation}
{\mathcal{B}}_{\mu}(k,{\mathcal{E}}_{\min }(k);x,y):=\mu |\hat{v}(x)|^{\frac{1}{2}}\widehat{\mathcal{R}}_{0}(k,{%
\mathcal{E}}_{\min }(k);x-y)|\hat{v}(y)|^{\frac{1}{2}},\,x\,,y\in
\mathbb{Z}^{d} \label{limkernel_function}
\end{equation}%
as follows
\begin{equation}
(\mathbb{B}_{\mu}(k,{\mathcal{E}}_{\min}(k))\hat{f})(x)=\sum\limits_{y\in\Z^d}{\mathcal{B}}_{\mu}(k,{\mathcal{E}}_{\min
}(k);x,y)\hat{f}(y),\,\, \hat{f}\in\ell^{2,e}({\Z}^d).
\end{equation}
\end{definition}
Now we formulate some properties of the \emph{generalized
Birman-Schwinger operator}.
\begin{lemma}\label{kernelBS}
Let $d\geq3$ and  assume Hypothesis \ref{hypothesis}.
\begin{enumerate}
\item[(i)] For any  $x\in \Z^d$, the kernel function $\widehat{\mathcal{R}}_0(\cdot,\cE_{\min }(\cdot);x)$
is continuous in $k\in \mathbb{G}$.

\item[(ii)] For any  $x,y \in \Z^d$, the kernel function $\cB_{\mu}(\cdot,\cE_{\min }(\cdot);x,y)$
is continuous in $k\in \mathbb{G}$.

\item[(iii)] For any $k\in \mathbb{G}$, the operator
$\mathbb{B}_{\mu}(k,\cE_{\min}(k))$, acting in $\ell^{2,e}({\Z}^d)$,
belongs to the Hilbert-Schmidt class and the map  $k \in
\mathbb{G}\mapsto \mathbb{B}_{\mu}(k,\cE_{\min}(k))$ is continuous.
\end{enumerate}
\end{lemma}
\begin{remark}
Lemma \ref {kernelBS} yields that for each $d\geq3$ the kernel
function $\cB_{\mu}(k,\cE_{\min }(k);x,y)$ defines a
\text{generalized Birman-Schwinger operator}
$\mathbb{B}_{\mu}(k,\cE_{\min }(k))$ on $\ell^{2,e}(\Z^d),$ which is
non-negative, compact and hence self-adjoint Hilbert-Schmidt
operator.
\end{remark}
Let $\ell_{0}(\Z^d)$ be the Banach space of all functions defined on
$\Z^d$ and  vanishing at infinity.

\subsection{Regular or singular point}

Since for each $k \in \mathbb G$ the operator
$\mathbb{B}_{\mu}(k,\cE_{\min }(k))$ is compact (self-adjoint), only
one of the following two cases may happen
\begin{enumerate}
\item[(i)] The number $1$ is an eigenvalue for $\mathbb{B}_{\mu}(k,\cE_{\min
}(k))$.

\item[(ii)] The number $1$ is not an eigenvalue for $\mathbb{B}_{\mu}(k,\cE_{\min
}(k))$.
\end{enumerate}
\begin{definition}\label{def_singular_point}
Let $d\geq3.$ For each $k\in \mathbb G$ the threshold
$z=\cE_{\min}(k)$ of the essential spectrum
$\sigma_{\text{ess}}(\widehat{H}_{\mu}(k))$ is called a singular
point of multiplicity $n$ (resp. regular point) of the operator
$\widehat{H}_{\mu}(k)$, if the number $1$ is an eigenvalue of
multiplicity $n$ (resp. not an eigenvalue) for the operator
$\mathbb{B}_{\mu}(k,\cE_{\min}(k))$.
\end{definition}

\begin{remark}\label{def_regular_point}
If the threshold $\cE_{\min}(k)$ is a regular point of the
essential spectrum of $\widehat{H}_{\mu}(k)$, then Theorem \ref{General_BS} %
yields that the equation $\widehat{H}_{\mu}(k)\hat{f}=
\cE_{\min}(k)\hat{f},\hat{f} \in \ell^{2,e}(\Z^d)$ has only the
trivial solution and the number of eigenvalues of the operator
$\widehat{H}_{\mu}(k)$ below the threshold $\cE_{\min}(k)$ remains
unchanged under small perturbations of $k\in \mathbb G$ and $\mu
\widehat{V}$ (see Theorem \ref{regular_point}).
\end{remark}

\begin{remark}\label{virtual_eigen} Notice that the threshold singular point is either
a threshold resonance or a threshold eigenvalue of
$\widehat{H}_{\mu}(k)$ and our definition of the threshold resonance
is the direct analogue of those, which have been introduced in the
continuous and lattice cases (see, e.g., \cite{ALMM2006},
\cite{Sobolev1993} and references therein).
\end{remark}

In the following theorem the Birman--Schwinger principle is extended
to $z=\cE_{\min}(k)$.

\begin{theorem} \label{General_BS}
Assume Hypothesis \ref{hypothesis}. Then the following statements
hold:

\begin{enumerate}
\item[(i)] Let $d\geq3.$ If $\hat{f}\in \ell_{0}(\Z^d)$ solves $\widehat{H}_{\mu}(k) \hat{f} = \cE_{\min }(k)
\hat{f}$, \,then $\hat{\psi}:=|\widehat{V}|^{1/2}\hat{f} \in
\ell^{2,e} (\Z^d)$ solves $\hat{\psi} = \mathbb{B}_{\mu}(k,\cE_{\min
}(k))\hat{\psi}$.

\item[(ii)] Let $d=3 \,,\,4.$ If $\hat{\psi} \in \ell^{2,e}(\Z^d)$ solves $\hat{\psi} = \mathbb{B}_{\mu}(k,\cE_{\min }(k))
\hat{\psi}$,
\,then\\
 $\hat{f}:= \widehat{R}_0(k, \cE_{\min }(k)) |\widehat{V}|^{1/2}\hat{\psi} \in
\ell_{0}(\Z^d)$ solves $\widehat{H}_{\mu}(k)\hat{f} =\cE_{\min
}(k)\hat{f}$.

\item[(iii)] Let $d\geq5.$ If $\hat{\psi} \in \ell^{2,e}
(\Z^d)$ solves $\hat{\psi} = \mathbb{B}_{\mu}(k,\cE_{\min }(k)) \hat{\psi}$, \,then\\
$\hat{f}:= \widehat{R}_0(k, \cE_{\min }(k))
|\widehat{V}|^{1/2}\hat{\psi} \in \ell^{2,e}(\Z^d)$ solves
$\widehat{H}_{\mu}(k)\hat{f}=\cE_{\min }(k)\hat{f}$.

\item[(iv)]  Let $d\geq3$. The threshold $\cE_{\min }(k)$ is a singular point
(a resonance or an eigenvalue) of the essential spectrum of
$\widehat{H}_{\mu}(k)$ with multiplicity $m$ if and only if the
number $1$ is an eigenvalue of $\mathbb{B}_{\mu}(k,\cE_{\min }(k))$
with multiplicity $m$.

\item[(v)] Let $d\geq3.$ Counting multiplicities, the number
 of eigenvalues of $\widehat{H}_{\mu}(k)$ less than $\cE_{\min }(k)$
equals to the number of eigenvalues of $\mathbb{B}_{\mu}(k,\cE_{\min
}(k))$ greater than $1,$ i.e.,
\begin{equation} \label{BS_equality}
\mathcal{N}_{-}(\cE_{\min }(k),\widehat{H}_{\mu}(k))
=\mathcal{N}_{+}(1,\mathbb{B}_{\mu}(k,\cE_{\min }(k))).
\end{equation}
\end{enumerate}
\end{theorem}
\begin{remark}\label{rem_singular_point}
Let the threshold $z=\cE_{\min}(k),\,k\in \mathbb{G}$ be a singular
point of the essential spectrum of $\widehat{H}_{\mu}(k)$, i.e., the
equation $\mathbb{B}_{\mu}(k,\cE_{\min}(k))\hat{\psi}=\hat{\psi}$
has a non-trivial solution  $\hat{\psi} \in\ell^{2,e}(\Z^d)$ and
wherein the function
\begin{equation}\label{has_form}
\hat{f}(x)=\sum\limits_{y\in\Z^d} \int\limits_{\T^d}
\frac{e^{i(p,x-y)}
\eta(\mathrm{d}p)}{\cE_k(p)-\cE_{\min}(k)}|\hat{v}(y)|^{\frac{1}{2}}\hat{\psi}(y)
\end{equation}
is a non-trivial solution of the Schr\"odinger equation
$\widehat{H}_{\mu}(k)\hat{f} = \cE_{\min}(k)\hat{f}$.
\begin{enumerate}
\item[(i)] If $d=3 \,\text {or}\,\, 4$, then the solution
$\hat{f}$ of the equation $\widehat{H}_{\mu}(k)\hat{f} =
\cE_{\min}(k)\hat{f}$ in \eqref{has_form} belongs to
$\ell_{0}(\Z^d)\setminus \ell^{2,e}(\Z^d)$.
\item[(ii)] If $d\geq 5$, then the solution $\hat{f}$ in \eqref{has_form}
belongs to $\ell^{2,e}(\Z^d)$ and hence the singular point
$\cE_{\min}(k)$ is an eigenvalue of the operator
$\widehat{H}_{\mu}(k)$.
\end{enumerate}
\end{remark}
\begin{definition} \label{def_virtual_level}
In the case (i) of Remark \ref{rem_singular_point}, the singular
point $\cE_{\min}(k),k\in \mathbb G$ is called a threshold resonance
(virtual level)  for $\widehat{H}_{\mu}(k)$.
\end{definition}

\subsection{Example. The discrete Laplacian.}\label{Lapl_degenerate}
The dispersion relation $\varepsilon$ associated to the discrete
Laplacian $\Delta_{\Z^d}$ is given as
\begin{equation}
\label{Lapl} \varepsilon(p)=\sum_{j=1}^{d}[1-\cos
p_j],\,\,p=(p_1,...,p_d)\in \T^d
\end{equation}
and hence it is a Morse function satisfying the Hypothesis
\ref{hypothesis}. The corresponding two-particle dispersion relation
$\cE_{k}(\cdot)$ is of the form
\begin{equation} \label{Lapl}
\cE_{k}(p)=2\sum_{j=1}^{d}[1-\cos\frac{k_j}{2}\cos
p_j],\,\,p=(p_1,...,p_d)\in \T^d.
\end{equation}
The function $\cE_{k}(\cdot)$ can be degenerate only for $k\in
\Pi_n,$ where
\begin{equation}\label{setPi}
\Pi_n=\left\{k\in\T^d:\hspace{1mm} \text{$n$ coordinates of $k$ is
equal to $\pi $}\right\},\quad 1\le n\le d.
\end{equation}
The set $\Pi_d$ consists of exactly one point $ (\pi,...,\pi)$ and
the set $\Pi_n \subset \T^d,\, 1\le n\le d$ is a surface (manifold)
of co-dimension $d-n$ and in this case the operator
$\widehat{H}_{\mu}(k)$ defined in  \eqref{two-particle}  becomes the
Schr\"odinger operator on $\Z^{d-n}.$ Note that the set $\mathbb
{G}=\T^d\setminus \{\cup_{i=1}^d \Pi_n\}$ is an open (maximal) set
in $\T^d$ satisfying the Hypothesis \ref{hypothesis} and the
function $\cE_{k}(p)$ is a Morse function for any $k\in \mathbb{G}$.

\begin{remark}
If the dispersion relation $\varepsilon(\cdot)$ is associated to the
discrete Laplacian $\Delta_{\Z^d}$, then for the Schr\"odinger
operator $\widehat{H}_{\mu}(k)$, given in Eq. \eqref{two-particle},
the region $\mathbb{G}$ in (ii) of Lemma \ref{region} can be defined
precisely as
 $\mathbb {G}=\T^d\setminus \{\cup_{i=1}^d \Pi_n\}$.
\end{remark}

\section{Statement of the main results}
The first result is the existence of the eigenvalues for each
nonzero potential $\mu \hat{v}$ and quasi-momentum $k\in
\mathbb{T}^{d}$ for  $d=1,2$.

\begin{theorem}\label{existence}
Let $d=1, 2.$ Assume the Hypothesis \ref{hypothesis} and
$\hat{v}\neq0.$ Then for any $k\in \mathbb{T}^d$ and $\mu>0$ the
operator $\widehat{H}_{\mu}(k)$ has an eigenvalue $z_{\mu}(k)$ below
the threshold $\cE_{\min}(k)$ of the essential spectrum
$\sigma_{ess}(\widehat{H}_{\mu}(k))$.
\end{theorem}
The next results are on dependence of the eigenvalues of
quasi-momentum $k\in \T^d$ and the conservation of the number of
eigenvalues, which lies below the essential spectrum of the operator
$\widehat{H}_{1}(0)=\widehat{H}_0(0)+\widehat{V}$ for all non-zero
$k$.

\begin{theorem}\label{Inequality1}
Let $d\geq1.$ Assume that $\varepsilon$ be conditionally negative
definite function on $\T^d$ and the inequality
$\widehat{H}_{1}(0)=\widehat{H}_0(0)+\widehat{V}\geq z_0I$ holds for
$z_0<\cE_{\min}(0)$. Then for all non-zero $k\in \T^d$, the strict
inequality $\widehat{H}_{1}(k)=\widehat{H}_0(k)+\widehat{V}>z_0I$
holds.
\end{theorem}
\begin{corollary}\label{Inequality2}
Let $d\geq1$. Let $\varepsilon$ be conditionally negative definite
functin on $\T^d$ and for any $k\in \T^d$ the numbers
$z_1(k)\le...\le z_m(k)$  be eigenvalues  of the operator
$\widehat{H}_\mu(k)$ (counting multiplicities) lying below
$\cE_{\min}(k)$. Then for any nonzero $k\in \T^d$, the inequalities
\begin{equation*}
z_j(0)<z_j(k)\,,j=1,...,m
\end{equation*}hold.
\end{corollary}\label{rem_regular_point}

\begin{theorem}\label{Inequalities}
Let $d\geq1$. Let $\varepsilon$ be conditionally negative definite
on $\T^d$ and  the numbers $z_1(0)\le...\le z_m(0)$ be $m$
eigenvalues
 of the operator $\widehat{H}_1(0)$ (counting multiplicities) lying below
$\cE_{\min}(0)$. Then there exists $z_0$, $z_0<\cE_{\min}(0)$ such
that for any nonzero $k\in \T^d$, the operator $\widehat{H}_1(k)$
has at least $m$ eigenvalues $z_1(k)\le...\le z_m(k)$ (counting
multiplicities) satisfying the inequalities
\begin{equation*}
z_j(k)<z_0+[\cE_{\min}(k)-\cE_{\min}(0)]\,,j=1,...,m.
\end{equation*}
\end{theorem}

Theorems \ref{Inequality1} and \ref{Inequalities} yield the
following corollary.

\begin{corollary}\label{Inequality3} Assume the assumptions of Theorem
\ref{Inequalities}. Then for any nonzero $k\in \mathbb{T}^d$
\begin{equation*}
z_j(0)<z_j(k)< z_j(0)+[\cE_{\min}(k)-\cE_{\min}(0)],\,\,j=1,...,m.
\end{equation*}
\end{corollary}
Now we prove, for $d\geq 3$, the existence of eigenvalues of
$\widehat{H}_\mu(k)$ for \textit{large} potentials $\mu \hat{v}$ and
each $k\in \mathbb{G}$. The next results precisely describe the
emission of eigenvalues from the essential spectrum
$\sigma_{ess}(\widehat{H}_\mu(k))$ (sf, \cite{KlausSimon1980}).

\begin{theorem}\label{existence(d>3)}
Let $d\geq3.$ Assume Hypothesis \ref{hypothesis} and that for some
$\mu>0$ and $k \in \mathbb{G}$, the inequality
\begin{equation}\label{maximum}
\mu\max_{x\in\Z^d}|\hat{v}(x)|\int\limits_{\T^d}
\frac{\eta(\mathrm{d}q)}{\cE_{k}(q)-\cE_{\min}(k)}>1\\
\end{equation} holds. Then  the operator $\widehat{H}_{\mu}(k)$ has an
eigenvalue $z_{\mu}(k)$ below the threshold $\cE_{\min}(k)$.
\end{theorem}

\begin{remark}\label{remark(d>3)}
Let $d\geq3$. Under the assumptions of Theorem \ref{existence(d>3)},
the integral at the left-hand side of the inequality \eqref{maximum}
is  continuous function in $k\in \mathbb{G}$ (see Lemma
\ref{kernelBS}). Hence, for any fixed $k_0 \in \mathbb{G}$  there
exists a neighborhood $\mathrm{G}(k_0)\subseteq \mathbb{G}$ \, of
$k_0$ such that for all $k\in \mathrm{G}(k_0)$ the inequality
\begin{equation*}
\mu\max_{x\in\Z^d}|\hat{v}(x)|\int\limits_{\T^d}\frac{\eta(\mathrm{d}q)}{\cE_{k}(q)-\cE_{\min}(k)}>1
\end{equation*}
holds. Therefore, the operator $\widehat{H}_{\mu}(k)$ has an
eigenvalue $z_{\mu}(k)$ below the threshold $\cE_{\min}(k)$ for all
$k \in \mathrm{G}(k_0)$.
\end{remark}
The following theorem states that for smooth dispersion relations
and absolutely convergent potentials on $\Z^d$, the Schr\"odinger
operators $\widehat{H}_{\mu}(k)$ have only finitely many eigenvalues
below the essential spectrum.

\begin{theorem}\label{finite}
Let $d\ge 3$. Assume  the Hypothesis \ref{hypothesis}. Then for each
$\mu>0$ and $\hat{v}\neq0$, the number of eigenvalues of the
operator $\widehat{H}_\mu(k), \,k \in \mathbb{G}$ lying below the
threshold $\cE_{\min}(k)$  is finite.
\end{theorem}

The following results describe the sets of the coupling constants
$\mu>0$ and the quasi-momenta $k\in \mathbb{G}$, for which
$\cE_{\min}(k)$  is a regular or a singular point of the essential
spectrum of $\widehat{H}_\mu(k)$.

\begin{theorem}\label{regular_point}
Assume $d\geq3$ and  Hypothesis \ref{hypothesis}. Let for some
$\mu_0>0$ and $k_0\in \mathbb{G}$ the threshold $\cE_{\min}(k_0)$ is
a regular point of the essential spectrum of the operator
$\widehat{H}_{\mu_0}(k_0)=\widehat{H}_0(k_0)+\mu_0 \widehat{V}$.
Then there exist neighborhoods $\mathrm{U}(\mu_0)\subset\R$ and
$\mathrm{G}(k_0)\subset\mathbb{G}$ of $\mu_0$ and $k_0\in
\mathbb{G}$, respectively, such that for all
$\mu\in\mathrm{U}(\mu_0)$ and $k\in\mathrm{G}(k_0)$, the number of
eigenvalues of operator $\widehat{H}_{\mu}(k)$ below the threshold
$\cE_{\min}(k)$ remains unchanged.
\end{theorem}

\begin{corollary}\label{rem_regular_point}
The set $\mathrm{U}\subset\R_+$ of all $\mu\in \R_+$ resp.
$\mathrm{G}\subset\mathbb{G}$\, of all $k\in \mathbb{G}$,\, for
which $\cE_{\min}(k)$, the threshold of the essential spectrum is
regular point of the operator $\widehat{H}_{\mu}(k)$ is an open set
in $\R_+$ resp. in $\mathbb{G}$, i.e., there exist intervals
$\mathrm{U}_\alpha\subset\R_+$ resp. connected components
$\mathrm{G}_\alpha \subset \mathbb{G}$, $\alpha =1,2,...$, such that
$\mathrm{U}=\cup_{\alpha}\mathrm{U}_\alpha$ resp.
$\mathrm{G}=\cup_{\alpha}\mathrm{G}_\alpha$. Consequently, the sets
$\R_+\setminus \mathrm{U}$  and $\mathbb{G}\setminus \mathrm{G}$
 consist of the points $\mu\in
\R_+$ and  quasi-momenta $k\in \mathbb{G}$, respectively, for which
$\cE_{\min}(k)$, the threshold of the essential spectrum is a
singular point of the operator $\widehat{H}_{\mu}(k)$.
\end{corollary}

\subsection{The threshold is a singular point}
Let $d\ge3$. Now, we study the emission of eigenvalues of the
operator $\widehat{H}_1(k),k\in \mathbb{G}$ from $\cE_{\min}(k)$,
the bottom of the essential spectrum
$\sigma_{ess}(\widehat{H}_1(k))$, when the threshold is a singular
point of the essential spectrum $\sigma_{ess}(\widehat{H}_1(k))$.

Recall that for any $\hat{\psi}\in \ell^{2,e}(\Z^d)$, the Fourier
transform $\cF \circ |\widehat{V}|^{1/2}\hat{\psi}$ of the function
$|\widehat{V}|^{1/2}\hat{\psi}$ is continuous in $p\in \T^d$.

For each $k_0\in \mathbb{G},\,s\in\Z^d$ and $\hat{\psi}\in
\ell^{2,e}(\Z^d)$ under Hypothesis \ref{hypothesis}, the function
\begin{equation} \label{positive}
\cC_s(k,k_0;\psi)=\int\limits_{\T^d}\frac{[1-\cos(p,s)]| \cF \circ
|\widehat{V}|^{1/2}\hat{\psi}|^2 \eta(\mathrm{d}p)}
{(\cE_{k}(p)-\cE_{\min }(k))(\cE_{k_0}(p)-\cE_{\min }(k_0))}\geq 0\\
\end{equation}
is continuous in  $k\in\mathbb{G}$ and can be estimated by a
constant not depending on $s\in\Z^d$.

\begin{remark} \label{eigenvalue}
$(i)$ Let for some $k_0\in \mathbb{G}$, $\cE_{\min}(k_0)$, the
threshold of the essential spectrum of $\widehat{H}_1(k_0)$ be a
singular point of multiplicity $n\ge1$, i.e., the number $1$ is an
eigenvalue for $\mathbb{B}_1(k_0,\cE_{\min }(k_0))$ of multiplicity
$n\geq1$. Let $\cH_n$ be an $n$-- dimensional subspace of the
operator $\mathbb{B}_1(k_0,\cE_{\min }(k_0))$, associated to the
eigenvalue $1$. Then for any non-zero $\hat{\psi}\in \cH_n$, the
inequality $\cC_s(k,k_0;\hat{\psi})>0,\,s\in \Z^d$ holds.

$(ii)$ Let the threshold $\cE_{\min }(k_0), k_0\in \mathbb{G}$ be a
regular point of $\widehat{H}_\mu(k_0),$ i.e., the equation
$\mathbb{B}_\mu(k_0,\cE_{\min }(k_0))\hat{\psi}=\hat{\psi}$ has only
the trivial solution $\hat{\psi}=0 \in \ell^{2,e}(\Z^d)$. Then for
all $s\in \Z^d$, the equality $\cC_s(k,k_0;\psi)=0$ holds.

\end{remark}

We define the function $\mathrm{L}(k,k_0;\psi)$ on
$\mathbb{G}\subset\T^d$ as
\begin{equation}\label{function}
\mathrm{L}(k,k_0;\psi)=2\sum_{s\in\Z^d,s\ne0}  \hat{\varepsilon}(s)
[\cos (\frac{k}{2},s)-\cos(\frac{k_0}{2},s)]\cdot \cC_s(k,k_0;\psi).
\end{equation}
Since $\hat{\varepsilon} \in \ell^{1}(\Z^d)$, the series in the
right hand side of  Eq. \eqref{function} is absolutely convergent
and defines a continuous function in $k\in\mathbb{G}$.

For any fixed $k_0\in \mathbb{G}$ and non-zero $\hat{\psi}\in
\ell^{2,e}(\Z^d)$ we define the following sets
\begin{align*}
&M_>(k_0;\hat{\psi}):=\{k\in\mathbb{G}:\mathrm{L}(k,k_0;\hat{\psi})>0\}\subset \mathbb{G}, \\
&M_=(k_0;\hat{\psi}):=\{k\in\mathbb{G}:\mathrm{L}(k,k_0;\hat{\psi})=0\}\subset \mathbb{G}, \\
&M_<(k_0;\hat{\psi}):=\{k\in\mathbb{G}:\mathrm{L}(k,k_0;\hat{\psi})<0\}\subset
\mathbb{G}.
\end{align*}
Repeating the argument of the proof of Lemma \ref{kernelBS}, one can
show that the map $k\in
\mathbb{G}\mapsto\mathrm{L}(k,k_0;\hat{\psi})$ is continuous. The
sets $M_>(k_0;\hat{\psi})$ and $M_<(k_0;\hat{\psi})$ are open
subsets of $\mathbb{G} \subset\T^d$ and
\begin{equation*}
\mathbb{G}=M_>(k_0;\hat{\psi})\cup M_{=}(k_0;\hat{\psi})\cup
M_{<}(k_0;\hat{\psi}).
\end{equation*}
For the $n$ be dimensional subspace $\cH_n$, associated to the
eigenvalue
$1$ of the  operator $\mathbb{B}_\mu(k_0,\cE_{\min }(k_0))$,\\
$k_0\in \mathbb{G}$, we define the following subsets of $\mathbb{G}$
\begin{align*}
&\mathrm{M}_>(k_0;\cH_n)=\bigcup_{\hat{\psi}\in \cH_n,||\hat{\psi}||=1} M_>(k_0;\hat{\psi}),\\
&\mathrm{M}_=(k_0;\cH_n)=\bigcup_{\hat{\psi}\in \cH_n,||\hat{\psi}||=1} M_=(k_0;\hat{\psi}),\\
&\cM_>(k_0;\cH_n)=\bigcap_{\hat{\psi}\in \cH_n,||\hat{\psi}||=1} M_>(k_0;\hat{\psi}),\\
&\cM_=(k_0;\cH_n)=\bigcap_{\hat{\psi}\in \cH_n,||\hat{\psi}||=1} M_=(k_0;\hat{\psi}),\\
&\cM_<(k_0;\ell^{2,e}(\Z^d))=\bigcap_{\hat{\psi}\in
\ell^{2,e}(\Z^d),||\hat{\psi}||=1} M_<(k_0;\hat{\psi}).
\end{align*}

The following results precisely describes the emission of
eigenvalues at the threshold $\cE_{\min}(k)$ depending on the
quasi-momentum $k\in \mathbb{G}$ and the potential $\widehat{V}$.

\begin{theorem}\label{singular_point}
Let $d\geq3$. Assume Hypothesis \ref{hypothesis} and that for
$k_0\in \mathbb G$ the threshold $\cE_{\min}(k_0)$ is a singular
point (of multiplicity $n=1,2,...$) for the operator
$\widehat{H}_{1}(k_0)=\widehat{H}_0(k_0)+\widehat{V}$ satisfying the
inequality $\widehat{H}_{1}(k_0)\geq \cE_{\min}(k_0)I$. Then:
\begin{enumerate}
\item[(i)] For any  $k\in\mathrm {M}_{>}(k_0;\cH_{n})$, the operator $\widehat{H}_{1}(k)$ has an
eigenvalue below the threshold $\cE_{\min}(k)$.

\item[(ii)] For any $k\in
\mathrm{M}_{=}(k_0;\cH_{n})$,  the threshold $\cE_{\min}(k)$ is a
{\it singular point} of the essential spectrum of
$\widehat{H}_{\mu}(k)$.

\item[(iii)] For any  $k\in
\cM_>(k_0;\cH_{n})$,  the operator $\widehat{H}_{1}(k)$ has at least
$n\geq1$ eigenvalues $z_1(k),...,z_n(k)$ (counting multiplicities)
below the threshold $\cE_{\min}(k)$.

\item[(iv)] For any $k\in
\cM_=(k_0;\cH_{n})$,  the threshold $\cE_{\min}(k)$ is a singular
point for the essential spectrum of $\widehat{H}_{1}(k)$ (with
multiplicity $n$).

\item[(v)] For any  $k\in
\cM_<(k_0;\ell^{2,e}(\Z^d))$, the operator $\widehat{H}_{1}(k)$ has
no eigenvalues below the threshold $\cE_{\min}(k)$.
\end{enumerate}
\end{theorem}
\begin{corollary}\label{perturbation(mu)}
Let the assumptions of Theorem \ref {singular_point} are fulfilled.
Then:
\begin{enumerate}
\item[(i)] For any $\mu>1$ and $k\in \mathrm{M}_=(k_0;\cH_{n})$, the operator
$\widehat{H}_{\mu}(k)=\widehat{H}_0(k)+\mu \widehat{V}$ has at least
one eigenvalue below the threshold $\cE_{\min}(k)$.
\item[(ii)] For any $\mu>1$ and $k\in \cM_=(k_0;\cH_{n})\cup \cM_>(k_0;\cH_{n})$,  the operator
$\widehat{H}_{\mu}(k)=\widehat{H}_0(k)+\mu \widehat{V}$ has at least
$n$ eigenvalues (counting multiplicities) below the threshold
$\cE_{\min}(k)$.
\end{enumerate}
\end{corollary}
\begin{remark}\label{Rem_regular_point}
Theorem \ref{singular_point} yields that for any $k_{=}\in
\mathrm{M}_=(k_0;\cH_{n})$, there is an open ball
$\mathrm{B}(k_{=})$ with the center $k_{=}\in
\mathrm{M}_=(k_0;\cH_{n})$ such that the sets $\mathrm{B}(k_{=})\cap
\mathrm{M}_{>}(k_0;\cH_{n})$ and $\mathrm{B}(k_{=})\cap
\mathrm{M}_{<}(k_0;\cH_{n})$ are non-empty. Thus, for any $k_{=}\in
\mathrm{M}_=(k_0;\cH_{n})$ there are open sets
$\mathfrak{N}_{>}\subset \mathrm{M}_{>}(k_0;\cH_{n})$ and
$\mathfrak{N}_{<}\subset \mathrm{M}_{>}(k_0;\cH_{n})$ such that for
any $k\in \mathfrak{N}_{>}$ the operator $\widehat{H}_{1}(k)$ has an
eigenvalue below $\cE_{\min}(k)$ and it has no eigenvalues for $k\in
\mathfrak{N}_{<}$.
\end{remark}

\begin{corollary}\label{greater}
Let the assumptions of Theorem \ref {singular_point} are fulfilled.
Then for each $k_{=} \in \mathrm{M}_{=}(k_0;\cH_{n})$, there exists
$k_{>}\in \mathrm{M}_{>}(k_0;\cH_{n})$ such that the number of
eigenvalues of \, $\widehat{H}_1(k_{>})$\, below $\cE_{\min}(k_{>})$
is greater than the number of eigenvalues (counting multiplicities)
of \ $\widehat{H}_1(k_{=})$\, below $\cE_{\min}(k_{=})$.
\end{corollary}

\begin{remark} \label{eigenvalue}
Let the operator  $\widehat{H}_{1}(k_0), k_0\in \mathbb{G}$ has
$m\ge1$ eigenvalues (counting multiplicities) lying below the
threshold $\cE_{\min }(k_0).$ Then by the Theorem \ref{General_BS},
the operator $\mathbb{B}_1(k_0,\cE_{\min }(k_0))$ has $m$
eigenvalues (counting multiplicities) $\lambda_1\ge...\ge\lambda_m
>1$. Let $\cH^{eig}_m$ be an $m$- dimensional subspace, spanned to
the eigenfunctions of the operator $\mathbb{B}_1(k_0,\cE_{\min
}(k_0))$,\,$k_0\in \mathbb{G}$ associated to
$\lambda_1,...,\lambda_m$. Then for any non-zero $\psi \in
\cH^{eig}_m$, the inequality $\cC_s(k,k_0;\psi)>0,\,s\in \Z^d$
holds, where $\cC_s(k,k_0;\psi)$ is defined in \eqref{positive}.
\end{remark}

For the $m$- dimensional subspace $\cH^{eig}_m$, associated to the
eigenvalues lying below the threshold $\cE_{\min }(k_0)$ of
$\widehat{H}_{1}(k_0)$, we define the subset
$\cM_>(k_0;\cH^{eig}_m)\subset\mathbb{G}$ as
\begin{equation*}
\cM_>(k_0;\cH^{eig}_m)=\bigcap_{\hat{\psi}\in
\cH^{eig}_m,||\hat{\psi}||=1} M_>(k_0;\hat{\psi}).
\end{equation*}

Theorems \ref{Inequalities} and \ref{singular_point} yield that the
total number of both the multiplicity of singular point
$\cE_{\min}(k)$ and the eigenvalues (counting multiplicities) below
the threshold $\cE_{\min}(k),k \in \mathbb G$ of\\
$\widehat{H}_{\mu}(k)=\widehat{H}_0(k)+\mu \widehat{V}$ is a
nondecreasing function in $k \in \mathbb G$  and $\mu>0$ in the
sense of the following Corollary.
\begin{corollary}\label{total_number}
Let $d\geq3.$ Assume Hypothesis \ref{hypothesis} and that
$\cE_{\min}(k_0), k_0 \in \mathbb G$  is a singular point of
multiplicity $n$ and the operator $\widehat{H}_\mu(k_0)$ has only
$m$ eigenvalues (counting multiplicities) below the threshold
$\cE_{\min}(k_0)$.
\begin{enumerate}
\item[(i)]If $k\in \cM_=(k_0;\cH_{n})\bigcap \cM_>(k_0;\cH_{m}^{eig})$, then for  all
$k\in\cM_>(k_0;\cH_{n})$  the operator $\widehat{H}_\mu(k)$ has at
least $n+m$ eigenvalues (counting multiplicities) lying below
$\cE_{\min}(k).$

\item[(ii)] For all $\mu>1$ and $k\in
\cM_=(k_0;\cH_{n})\bigcap \cM_>(k_0;\cH_{m}^{eig})$, the operator
$\widehat{H}_\mu(k)$ has at least $n+m$ eigenvalues (counting
multiplicities) lying below $\cE_{\min}(k)$.
\end{enumerate}
\end{corollary}

\section{The proof of the results}

{\bf The proof of Lemma \ref{kernelBS}} (i) Since
$\varepsilon(\cdot)$ satisfies  Hypothesis \ref{hypothesis} for any
$k\in\mathbb{G}$ the point $0\in \T^d,\,d\ge3$ is a non-degenerate
minimum of the function $\cE_{k}(\cdot)$.

According to the parametrical Morse lemma (see, e.g.,
\cite{Fedoryuk1987}, p. 113) on smooth functions, one concludes that
for any $k\in \mathbb{G}$ there exists $C^{(1)}(W_{\gamma}(0))$--
diffeomorphisms $\phi(k,\cdot): W_{\gamma}(0)\rightarrow \cU(0)$ of
the ball $W_{\gamma}(0)\subset \mathbb{R}^d$ to a neighborhood
$\cU(0)\subset \T^d $ of the point $0\in \T^d$, so that the function
$\cE_k(\phi(k,q))$ can be represented as
\begin{align*}
&\cE_k(\phi(k,q))= \cE_{\min}(k) + q^2=
\cE_{\min}(k)+q^2_1+...+q^2_d,\,\cE_{\min}(k)\in
C^{(3)}(\mathbb{G}).
\end{align*}
The Jacobian $J(k,\phi(k,q))$ of the mapping $\phi$ is continuous in
$(k,q)\in \mathbb{G}\times W_{\gamma}(0)$ and $J(k,0)>0$ for any
$k\in \mathbb{G}.$

Therefore, for any fixed $k\in \mathbb{G}$ and $x \in \Z^d$, the
kernel function (of resolvent)
$\widehat{\mathcal{R}}_0(k,{\mathcal{E}}_{\min }(k);x)$ can be
written as the sum $\widehat{\mathcal{R}}_0(k,{\mathcal{E}}_{\min
}(k);x)=\widehat{\mathcal{R}}^{(1)}_0(k,{\mathcal{E}}_{\min
}(k);x)+\widehat{\mathcal{R}}^{(2)}_0(k,{\mathcal{E}}_{\min
}(k);x)$, where
\begin{align*}
&\widehat{\mathcal{R}}^{(1)}_0(k,{\mathcal{E}}_{\min }(k);x)
=\int\limits_{\cU(0)}\frac{e^{i(p,x)}\eta(\mathrm{d}p)}{\cE_k(p)-\cE_{\min}(k)},\\\,\,
&\widehat{\mathcal{R}}^{(2)}_0(k,{\mathcal{E}}_{\min
}(k);x)=\int\limits_{\mathbb{T}^d\setminus
\cU(0)}\frac{e^{i(p,x)}\eta(\mathrm{d}p)} {\cE_k(p)-\cE_{\min}(k)}.
\end{align*}
By making a change of variables $p=\phi(k,q)$, we obtain
\begin{align*}
&\widehat{\mathcal{R}}_0^{(1)}(k,{\mathcal{E}}_{\min }(k);x)=
\int\limits_{W_{\gamma}
(0)}\frac{e^{i(\phi(k,q),x)}}{q^2_1+...+q^2_d}J(k,\phi(k,q))\eta(\mathrm{d}q).
\end{align*}
In the spheroidal coordinates, by denoting $q=r\omega$, it can be
rewritten as
\begin{gather}
\widehat{\mathcal{R}}_0^{(1)}(k,{\mathcal{E}}_{\min }(k);x)= \int
\limits_{0}^{\gamma} r^{d-3}dr
\int\limits_{\Omega_{d-1}}e^{i(\phi(k,r\omega),x)}
J(k,\phi(k,r\omega)) d\omega,
\end{gather}
where $S^{d-1}$ is the unit sphere in $\R^d$ and $d\omega$ is its
element.

Observe that the function
\begin{equation*}
\widehat{\mathcal{R}}^{(2)}_0(k,{\mathcal{E}}_{\min
}(k);x)=\int\limits_{\mathbb{T}^d\setminus
\cU(0)}\frac{e^{i(p,x)}\eta(\mathrm{d}p)} {\cE_k(p)-\cE_{\min}(k)}
\end{equation*}
is continuously differentiable in $k\in \mathbb{G}$. Since
$J(k,\phi(k,q))$ is continuous in $(k,q)\in \mathbb{G}\times
W_{\gamma}(0)$ the functions
$\widehat{\mathcal{R}}_0^{(1)}(k,{\mathcal{E}}_{\min }(k);x)$ and
$\widehat{\mathcal{R}}_{0}(k,{\mathcal{E}}_{\min
}(k);x)=\widehat{\mathcal{R}}_0^{(1)}(k,{\mathcal{E}}_{\min
}(k);x)+\widehat{\mathcal{R}}_0^{(2)}(k,{\mathcal{E}}_{\min }(k);x)$
are continuous in $k\in \mathbb{G}$.

(ii) For any fixed  $x,y \in \Z^d$, the continuity in $k\in
\mathbb{G}$ of the kernel function $\cB_{\mu}(k,\cE_{\min}(k); x,y)$
defined in \eqref{limkernel_function} can be proven by the same way
as (i) of Lemma \ref{kernelBS}.

(iii) For any $k\in \mathbb{G}$, the Hilbert-Schmidt norm
$||\mathbb{B}_{\mu}(k,\cE_{\min}(k))||_{2}$ of the operator
$\mathbb{B}_{\mu}(k,\cE_{\min}(k))$  can be estimated as
\begin{align*}
&||\mathbb{B}_{\mu}(k,\cE_{\min}(k))||^{2}_{2}=\sum\limits_{x,y\in\Z^d}|\cB_{\mu}(k,\cE_{\min}(k);x,y)|^2\\
& \le \mu^2 \sum\limits_{x,y\in\Z^d}
|\hat{v}(x)|\left|\int\limits_{\T^d}
\frac{e^{i(p,x-y)}\eta(\mathrm{d}p)}{\cE_k(p)-\cE_{\min}(k)}\right|^2|\hat{v}(y)|\\
&\le\mu^2\left[\int\limits_{\T^d}\frac{\eta(\mathrm{d}p)}{\cE_k(p)-\cE_{\min}(k)}\right]^2||\widehat{V}||_{\ell
^{1}(\mathbb{Z}^{d})}^2,
\end{align*}
i.e., the operator $\mathbb{B}_{\mu}(k,\cE_{\min}(k))$ belongs to
$\Sigma_{2}$. Since the operator norm
$||\mathbb{B}_{\mu}(k,\cE_{\min}(k))||$ satisfies the inequality
\begin{equation*}
||\mathbb{B}_{\mu}(k,\cE_{\min}(k))||^2\leq
||\mathbb{B}_{\mu}(k,\cE_{\min}(k))||^2_{2}
\end{equation*}
the operator $\mathbb{B}_{\mu}(k,\cE_{\min}(k))$ is continuous in $k
\in \mathbb{G}$.$\square$

{\bf The proof of Theorem \ref{General_BS}}. We only prove the items
(i), (ii) and (iii), since the case (iv) can be proven as in Lemma
2.2 of ref. \cite {BdsPL2018} and the case (v) similarly to Theorem
6 of ref. \cite{LakaevAlladust2014}.

Let $d\geq3$ and $\hat{v} \in \ell^1(\Z^d; \R_0^-)$.

(i) Note that $\hat{f} \in \ell_{0}(\Z^d)$ and $\hat{v} \in
\ell^1(\Z^d; \R_0^-)$ yield
$\hat{\psi}=|\widehat{V}|^{\frac{1}{2}}\hat{f} \in
\ell^{2,e}(\Z^d)$. Now let $\hat{f} \in\ell_{0}(\Z^d)$ be a solution
of $(\widehat{H}_{\mu}(k)-\cE_{\min}(k){I})\hat{f}=0$, i.e., the
equality
\begin{equation}\label{a1}
(\widehat{H}_0(k)-\cE_{\min}(k)I)\hat{f}=-\mu \widehat{V}\hat{f}
\end{equation} holds.
Then the conditions for $\hat{v}$ of Hypothesis \ref{hypothesis}
yields that $\widehat{V} \hat{f}\in\ell^1(\Z^d)$. Using the
definition of $\widehat{R}_0(k,\cE_{\min}(k))$  and the properties
of the Fourier transform, it can be shown that
\begin{align*}
[\widehat{R}_0(k,\cE_{\min}(k))][\widehat{H}_0(k)-\cE_{\min}(k)I]=I,
\end{align*}
i.e., the operator $ \widehat{R}_0(k,\cE_{\min}(k))$ is the
resolvent of $ \widehat{H}_{0}(k)$ at the threshold
$z=\cE_{\min}(k)$.

The equality \eqref{a1} yields the relation
\begin{align*}
\hat{f}=\widehat{R}_0(k,\cE_{\min}(k))\left(
\widehat{H}_{0}(k)-\cE_{\min}(k) I\right)\hat{f}= \mu
\widehat{R}_0(k,\cE_{\min}(k))|\widehat{V}|\hat{f}
\end{align*}
and
\begin{align*}
|\widehat{V}|^{\frac{1}{2}}\hat{f}=
\mu[|\widehat{V}|^{\frac{1}{2}}\widehat{R}_0(k,\cE_{\min}(k))|\widehat{V}|^{\frac{1}{2}}]
|\widehat{V}|^{\frac{1}{2}}\hat{f},
\end{align*}
i.e., $\hat{\psi}=|\widehat{V}|^{\frac{1}{2}}\hat{f}$ is a solution
of ${\mathbb{B}}_{\mu}(k,\cE_{\min}(k))\hat{\psi}=\hat{\psi}.$

(ii) Let $d=3 \,,\,4.$ Since  $\hat{v} \in \ell^1(\Z^d; \R_0^-)$ is
a non-positive function, for any $\hat{\psi}\in\ell^{2,e}(\Z^d)$,
the Cauchy--Schwarz inequality leads to the relation
$|\widehat{V}|^{\frac{1}{2}}\hat{\psi}\in\ell^{1}(\Z^d)$.

Since the function $\dfrac{1}{\cE_k(\cdot)-\cE_{\min}(k)}$ is
integrable on $\mathbb{T}^d$, the Riemann--Lebesgue lemma yields
\begin{equation}\label{q2}
\widehat{\mathcal
R}_0(k,\cE_{\min}(k);x):=\int\limits_{\T^d}\dfrac{\mathrm e^{\mathrm
i (p,x)} \eta(\mathrm dp)}{\cE_k(p)-\cE_{\min}(k)}\rightarrow 0
\,\,\,\,\mbox{as}\,\,\,\,|x|\rightarrow \infty.
\end{equation}
The inclusion $ |\widehat{V}|^{\frac{1}{2}}\hat{\psi}\in
\ell^{1}(\Z^d)$ and  \eqref{q2}  lead to
\begin{equation}
\hat{f}(x)=\sum_{y\in\Z^d} \widehat{\mathcal
R}_0(k,\cE_{\min}(k);y-x)
|\hat{v}(y)|^{\frac{1}{2}}\hat{\psi}(y)\rightarrow 0
\,\,\,\mbox{as}\,\,\, |x|\rightarrow \infty,
\end{equation} i.e., $\hat{f}\in\ell_{0}(\Z^d)$.

Now let $\hat{\psi} \in \ell^{2,e}(\Z^d)$ be a solution of the
equation
\begin{align}\label{solution}
\mu|\widehat{V}|^{\frac{1}{2}}\widehat{R}_0(k,\cE_{\min}(k))|\widehat{V}|^{\frac{1}{2}}\hat{\psi}=\hat{\psi}.
\end{align}
By denoting
\begin{equation}\label{q4}
\hat{f} = \widehat{R}_0(k,\cE_{\min}(k))
 |\widehat{V}|^{\frac{1}{2}}\hat{\psi}
\end{equation}
we have that $\hat{f}\in \ell_{0}(\Z^d)$. Equation \eqref{solution}
yields that
\begin{equation}\label{q3}
\mu|\widehat{V}|^{\frac{1}{2}}\hat{f}=\hat{\psi} \,\,\mbox{and}\,\,
\mu|\widehat{V}|\hat{f}= |\widehat{V}|^{\frac{1}{2}}\hat{\psi}.
\end{equation}
The equality \eqref{q4} and the second equality in  \eqref{q3} imply
that
\begin{equation}\label{q5}
\left(\widehat{H}_{0}(k)-\cE_{\min}(k) I\right)\hat{f} =-\mu\left(
\widehat{H}_{0}(k)-\cE_{\min}(k) I\right)
\widehat{R}_0(k,\cE_{\min}(k))\widehat{V}\hat{f}=-\mu
\widehat{V}\hat{f},
\end{equation}
i.e., \,$\hat{f}\in\ell_{0}(\Z^d)$ is a solution of the equation
$\widehat{H}_{\mu}(k)\hat{f}=\cE_{\min}(k)\hat{f}$.

(iii) Let $d\geq 5$. According to Lemma \ref{region} the inclusion
\begin{equation*}
\dfrac{1}{\cE_k(\cdot)-\cE_{\min}(k)}\in
L^{2,e}(\T^d,\eta)
\end{equation*}
holds and hence the Plancherel
theorem leads to
\begin{equation}\label{q22}
\widehat{\mathcal
R}_0(k,\cE_{\min}(k);x):=\int\limits_{\T^d}\dfrac{\mathrm e^{\mathrm
i (p,x)} \eta(\mathrm
dp)}{\cE_k(p)-\cE_{\min}(k)}\in\ell^{2,e}(\Z^d).
\end{equation}
Since for any $\hat{\psi}\in\ell^{2,e}(\Z^d)$ the relation
$|\widehat{V}|^{\frac{1}{2}}\hat{\psi}\in \ell^{1}(\Z^d)$ holds, the
Young convolution inequality \cite[IX.4]{RSII1979} and the relation
\eqref{q22} yield that
\begin{equation*}
\hat{f}=\widehat{R}_{0}(k,\cE_{\min}(k))
|\widehat{V}|^{\frac{1}{2}}\hat{\psi}\in \ell^{2,e}(\Z^d).
\end{equation*}
 The equality $\widehat{H}_{\mu}(k)\hat{f}=\cE_{\min}(k)\hat{f},\,\,
\hat{f}\in\ell^{2,e}(\Z^d)$ can be proven as the case above.
$\square$

{\bf Proof of Theorem \ref{existence}} Let $d=1,\, 2$ and
$\hat{v}\le0$ is not identically zero. Then the number
$\lambda={v}(s),\,s\in \Z^d$ is an eigenvalue of the operator
$\widehat{V}$ and the kronecker delta function  $
\hat{\psi}_{s}(x)=\delta_{sx}\in \ell^{2,e}(\Z^d)$ of the point
$s\in \Z^d$ is the associated eigenfunction.  Then for any
$z<\cE_{\min}(k)$
\begin{align}\label{relations}
&(\mathbb{B}_\mu(k,z)\hat{\psi}_{s},\hat{\psi}_{s})\\
&=\mu(|\widehat{V}|^\frac{1}{2}\widehat{R}_0(k,z)|\widehat{V}|^\frac{1}{2}\hat{\psi}_{s},\hat{\psi}_{s}) \nonumber \\
&=\mu(\widehat{R}_0(k,z)|\widehat{V}|^\frac{1}{2}\hat{\psi}_{s},|\widehat{V}|^\frac{1}{2}\hat{\psi}_{s})\nonumber \\
&=\mu\int\limits_{\T^d}
\frac{|\cF \circ |\widehat{V}|^{1/2}\hat{\psi}_s|^2\eta(\mathrm{d}p)}{\cE_k(p)-z}\nonumber \\
&=\mu\int\limits_{\T^d}
\frac{|\sum\limits_{x\in\Z^d}e^{i(p,x)}|\hat{v}(x)|^{\frac{1}{2}}\hat{\psi}_{s}(x)|^2\eta(\mathrm{d}p)}{\cE_k(p)-z}\nonumber\\
&=\mu|\hat{v}(s)|\int\limits_{\T^d}\frac{\eta(\mathrm{d}p)}{\cE_k(p)-z}.\nonumber
\end{align}
Since the last integral in  \eqref{relations} is monotonically
increasing and continuous function in $z\in
(-\infty,\cE_{\min}(k))$, similarly to the proof of Lemma
\ref{kernelBS}, one can show
\begin{equation*}
\lim_{z\rightarrow
\cE_{\min}(k)}\int\limits_{\T^d}\frac{\eta(\mathrm{d}p)}{\cE_k(p)-z}=\int\limits_{\T^d}
\frac{\eta(\mathrm{d}p)}{\cE_k(p)-\cE_{\min}(k)}=+\infty.
\end{equation*}
Therefore, for some $z<\cE_{\min}(k)$, the inequality
$(\mathbb{B}_\mu(k,z)\hat{\psi}_s,\hat{\psi}_s)>1$ holds, i.e., the
self-adjoint compact operator $\mathbb{B}_\mu(k,z)$ has an
eigenvalue in the interval $(1,+\infty)$. Then, Proposition \ref{BS}
yields that the operator $\widehat{H}_\mu(k)$ has an eigenvalue in
the interval $(-\infty,\cE_{\min}(k))$.$\square$

{\bf The proof of Theorem \ref{Inequality1}} can be proven
analogously to Theorem 1 of \cite{LakaevAlladust2014}.

{\bf Proof of Theorem \ref{Inequalities}}. Let $d=1,2$ and operator
$\widehat{H}_1(0),\,0\in \T^d$ has $m\geq 1$ eigenvalues $z_1(0)\le
...\le z_m(0)<\cE_{\min}(0)$ (counting multiplicities). Then for any
$z_0\in (z_m(0), \cE_{\min}(0))$, the Proposition \ref{BS} yields
the existence of $m-$ dimensional linear subspace  $\cH_m \subset
\ell^{2,e}(\Z^d)$ spanned to the eigenvectors, associated to the
eigenvalues $\lambda_1(0)\ge...\ge \lambda_m(0)>1$ (counting
multiplicities), of the operator $\mathbb{B}_1(0,z_0)$, such that
for any non-zero $\psi \in\cH_{m}$ the relations $|\cF \circ
|\widehat{V}|^{1/2}\hat{\psi}|\neq0$ and
\begin{align*}
&(\mathbb{B}_1(0,z_0)\hat{\psi},\hat{\psi})=
(|\widehat{V}|^\frac{1}{2}\widehat{R}_0(0,z_0)|\widehat{V}|^\frac{1}{2}\hat{\psi},\hat{\psi})=
(\widehat{R}_0(0,z_0)|\widehat{V}|^\frac{1}{2}\hat{\psi},|\widehat{V}|^\frac{1}{2}\hat{\psi})\\
&=\int\limits_{\T^d} \frac{|\cF \circ
|\widehat{V}|^{1/2}\hat{\psi}|^2\eta(\mathrm{d}p)}{\cE_{0}(p)-z_0}>(\hat{\psi},\hat{\psi})
\end{align*}
hold.

Since the function ${\varepsilon}$ is conditionally negative
definite, the minimum
$\cE_{\min}(k)=\min\limits_{p\in\T^d}{\cE}_k(p)=2\,
{\varepsilon}(\frac{k}{2}),\,k\in \T^d$ of the function ${\cE}_k(p)$
is attained at the point $p(k)=0$. Hence, for any non-zero $k\in
\T^d$ and $z_0\in (z_m(0),\, \cE_{\min}(0))$, the inequality
\begin{equation}\label{inequalityforEps}
\cE_0(p)-z_0>{\cE}_k(p)-[z_0+{\cE}_{\min
}(k)-\cE_{\min}(0)],\,\,p\in \T^d
\end{equation}
holds, which can be proven similarly to \cite[Lemma 5 ]{ALMM2006}.
Then for any non-zero $k\in \T^d$ and $\hat{\psi}\in\cH_{m}$ we have
\begin{align}\label{aaa3}
&\int\limits_{\T^d} \frac{|(\cF \circ
|\widehat{V}|^{1/2}\hat{\psi})(p)|^2\eta(\mathrm{d}p)}{\cE_{0}(p)-z_0}
<\int\limits_{\T^d}\frac{|(\cF \circ
|\widehat{V}|^{1/2}\hat{\psi})(p)|^2}{{\cE}_k(p)-[z_0+{\cE}_{\min
}(k)-{\cE}_{\min }(0)]}\eta(\mathrm{d}p).\\\no
\end{align}
Therefore, for any non-zero $k\in \T^d$ and $\hat{\psi} \in\cH_{m}$,
the relation
\begin{equation*}
(\mathbb{B}_\mu(k,z_0+[\cE_{\min}(k)-\cE_{\min}(0)]\hat{\psi},\hat{\psi})>(\hat{\psi},\hat{\psi})
\end{equation*}
holds, which implies that the compact operator
$\mathbb{B}_\mu(k,z_0+[\cE_{\min}(k)-\cE_{\min}(0)])$ has $m$
eigenvalues (counting multiplicity) greater than $1.$ Proposition
\ref{BS} gives that the operator $\widehat{H}_{\mu}(k),k \in
\mathbb{G}$ has $m$ eigenvalues $z_1(k)\le...\le z_m(k)$ (counting
multiplicity), which satisfy the relations
\begin{equation}\label{Inequalities2}
z_j(k)< z_0+[\cE_{\min}(k)-\cE_{\min}(0)]\,, j=1,...,m.
\end{equation}
This completes the proof of Theorem \ref{Inequalities}. $\square$.

{\bf Theorem \ref{existence(d>3)}} can be proven by the same way as
Theorem \ref{existence}.

{\bf Proof of  the Theorem \ref{finite}} follows from the equality
\begin{equation*} %
\mathcal{N}_{-}(\cE_{\min}(k),\widehat{H}_\mu(k))
=\mathcal{N}_{+}(1,\mathbb{B}_\mu(k,\cE_{\min}(k))),
\end{equation*}
which can be proven analogously to \cite[Theorem
6]{LakaevAlladust2014}.

{\bf Proof of the Theorem \ref{regular_point}.} If $\cE_{\min}(k_0)$
is a regular point, then the number $1$ is not an eigenvalue of the
operator $\mathbb{B}_\mu(k_0,\cE_{\min }(k_0)$ and hence there
exists a bounded operator
\begin{equation*}
(I-\mathbb{B}_\mu(k_0,\cE_{\min }(k_0)))^{-1}.
\end{equation*}

Since, the operator $\mathbb{B}_\mu(k,\cE_{\min }(k))$ is continuous
in $k\in \mathbb{G}$ (see (iii) of Lemma \ref{kernelBS}), there
exists a neighborhood  $\mathrm{G}(k_0) \subset \mathbb{G}$ of the
point $k_0\in \mathbb{G}$, such that the operator
$(I-\mathbb{B}_\mu(k,\cE_{\min}(k)))^{-1}$ exists and continuous in
$k\in \mathrm{G}(k_0).$ So, for any  $k\in \mathrm{G}(k_0)$, the
number of non-zero solutions of $\mathbb{B}_\mu(k,\cE_{\min
}(k))\hat{\psi}=\hat{\psi}$ in $\ell^{2,e}(\Z^d)$ remains unchanged.

The existence of a neighborhood of $\mu_0>0$ can be proven
similarly. $\square$

{\bf Proof of Theorem \ref{singular_point}}. We prove the cases
(iii) and (v) of Theorem \ref{singular_point}, since the cases (i),
(ii) and (iv) can be proven in the same way as (iii).

(iii) Under the assumptions of Theorem \ref{singular_point}, the
bottom $\cE_{\min }(k_0)$ is a singular point of
$\widehat{H}_1(k_0), k_0\in \mathbb{G}$ of multiplicity $n$.
Definition \ref{def_singular_point} of the singular point and the
Theorem \ref{General_BS} yield the existence a $n$- dimensional
subspace $\cH_{n} \subset \ell^{2,e}(\Z^d)$, associated to
eigenvalue $1$ and for any non-zero $\hat{\psi} \in\cH_{n}$ the
relations
\begin{align*}
&(\mathbb{B}_1(k_0,\cE_{\min}(k_0))\hat{\psi},\hat{\psi})=
(\widehat{R}_0(k_0,\cE_{\min}(k_0))|\widehat{V}|^\frac{1}{2}\hat{\psi},|\widehat{V}|^\frac{1}{2}\hat{\psi})=\int\limits_{\T^d}
\frac{|\cF \circ
|\widehat{V}|^{1/2}\hat{\psi}|^2\eta(\mathrm{d}p)}{\cE_{k_0}(p)-\cE_{\min}(k_0)}=(\hat{\psi},\hat{\psi})
\end{align*}
hold and therefore $|\cF \circ |\widehat{V}|^{1/2}\hat{\psi}|\neq0$.
Consequently, for any $k\in \cM_{>}(k_0;\cH_{n})$ and non-zero
$\hat{\psi}\in\cH_{n}$
\begin{align*}
&(\mathbb{B}_1(k,\cE_{\min}(k))\hat{\psi},\hat{\psi})-(\mathbb{B}_1(k_0,\cE_{\min }(k_0))\hat{\psi},\hat{\psi})\\
&=\int\limits_{\T^d}\frac{[(\cE_{k_0}(p)-\cE_{\min
}(k_0))-(\cE_{k}(p)-\cE_{\min }(k))]|\cF \circ
|\widehat{V}|^{1/2}\hat{\psi}|^2\eta(\mathrm{d}p)}
{[\cE_{k}(p)-\cE_{\min}(k))][\cE_{k_0}(p)-\cE_{\min }(k_0)]}.
\end{align*}
Applying the Fourier series expansion for $\varepsilon(p)$
\begin{equation*}
\cE_{k}(p)-\cE_{\min }(k)={\varepsilon}
(\frac{k}{2}+p)+{\varepsilon} (\frac{k}{2}-p)-2{\varepsilon}
(\frac{k}{2}) =2\sum_{s\in\Z^d} \hat{\varepsilon}(s) \cos
(\frac{k}{2},s)\cdot[\cos (p,s)-1],
\end{equation*}
gives the representation
\begin{align*}
&([\mathbb{B}_1(k,\cE_{\min}(k))-\mathbb{B}_1(k_0,\cE_{\min
}(k_0)]\hat{\psi},\hat{\psi})=2\sum_{s\in\Z^d,s\ne0}
\hat{\varepsilon}(s) [\cos (\frac{k}{2},s)-\cos (\frac{k_0}{2},s)]
\cC_s(k,k_0;\psi),
\end{align*}
where
\begin{align*}
&\cC_s(k_0,\psi;k)=\int\limits_{\T^d}\frac{[1-\cos(p,s))]|\cF \circ
|\widehat{V}|^{1/2}\hat{\psi}|^2\eta(\mathrm{d}p)}
{[\cE_{k}(p)-\cE_{\min}(k)][\cE_{k_0}(p)-\cE_{\min}(k_0)]}>0,\,s\in\Z^d\setminus\{0\}.
\end{align*}
So, the operator $\mathbb{B}_1(k,\cE_{\min}(k)),\,k\in
\mathbb{\cM}_{>}(k_0;\cH_{n})$ has $n$ eigenvalues greater than $1.$
Theorem \ref{General_BS} yields that the operator
$\widehat{H}_1(k),k\in \mathbb{\cM}_{>}(k_0;\cH_{n})$ has $n$
eigenvalues below the threshold $\cE_{\min}(k)$.

(v) For any $ k \in \cM_{<}(k_0;\ell^{2,e}(\Z^d))$ and non-zero
$\hat{\psi}\in \ell^{2,e}(\Z^d)$, we have
\begin{align*}
&\|\mathbb{B}_1(k,\cE_{\min }(k)\|\\
&= \sup_{\hat{\psi}\in \ell^{2,e}(\Z^d),\|\hat{\psi}\|=1 }(\mathbb{B}_1(k,\cE_{\min }(k))\hat{\psi},\hat{\psi})\\
&=\sup_{\hat{\psi}\in \ell^{2,e}(\Z^d),\|\hat{\psi}\|=1 }(|\widehat{V}|^\frac{1}{2}\widehat{R}_0(k,\cE_{\min }(k))
|\widehat{V}|^\frac{1}{2}\hat{\psi},\hat{\psi})\\
&=\sup_{\hat{\psi}\in \ell^{2,e}(\Z^d),\|\hat{\psi}\|=1 }(\widehat{R}_0(k,\cE_{\min }(k))|\widehat{V}|^\frac{1}{2}\hat{\psi},
|\widehat{V}|^\frac{1}{2}\hat{\psi})\\
&=\sup_{\hat{\psi}\in \ell^{2,e}(\Z^d),\|\hat{\psi}\|=1
}\int\limits_{\T^d}\frac{|\cF \circ
|\widehat{V}|^{1/2}\hat{\psi}|^2\eta(\mathrm{d}p)}
{\cE_{k}(p)-\cE_{\min }(k)}\\
&<\sup_{\hat{\psi}\in \ell^{2,e}(\Z^d),\|\hat{\psi}\|=1
}\int\limits_{\T^d}\frac{|\cF \circ
|\widehat{V}|^{1/2}\hat{\psi}|^2\eta(\mathrm{d}p)}
{\cE_{k_0}(p)-\cE_{\min }(k_0)}\\
&=\sup_{\hat{\psi}\in \ell^{2,e}(\Z^d),\|\hat{\psi}\|=1 }(\mathbb{B}_1(k_0,\cE_{\min }(k_0))\hat{\psi},\hat{\psi})\\
&=\|\mathbb{B}_1(k_0,\cE_{\min }(k_0)\|=1.
\end{align*}
Therefore, the non-negative compact operator
$\mathbb{B}_1(k,\cE_{\min}(k)),\,k\in
\mathbb{\cM}_{<}(k_0;\ell^{2,e}(\Z^d))$ may have only non-negative
eigenvalues smaller than $1$. Consequently, the Theorem
\ref{General_BS} yields that the operator $\widehat{H}_1(k),k\in
\mathbb{\cM}_{<}(k_0;\ell^{2,e}(\Z^d))$ has no eigenvalues below
$\cE_{\min}(k)$, the bottom of the essential spectrum of
$\widehat{H}_1(k)$.

{\bf Acknowledgments} S.N.Lakaev acknowledges to the Institute of
Applied Mathematics of the University Mainz for its kind hospitality
during his stay in the summer 2011 and also I.A.Ikromov and Sh.Yu.
Kholmatov for useful discussions and remarks. This research was
supported by the Foundation for Basic Research of the Republic of
Uzbe\-kistan (Grant No.OT-F4-66).

\end{document}